\documentclass[a4paper,11pt,fleqn]{article}
\usepackage[ruled,vlined]{algorithm2e}
\usepackage{amsmath,amsfonts,amsthm}
\usepackage{thm-restate,thmtools}
\usepackage{algorithmic}
\usepackage{nicefrac}
\usepackage[utf8]{inputenc}
\usepackage{color}
\usepackage{xcolor,colortbl}
\usepackage{bm,bbm}
\usepackage{soul}
\usepackage{mathtools}
\usepackage{nicefrac}  
\usepackage{booktabs}
\usepackage{adjustbox}
\usepackage{hyperref}
\usepackage{url}
\usepackage{cleveref}
\usepackage[flushleft]{threeparttable}
\usepackage{multirow}
\usepackage[shortlabels]{enumitem}
\usepackage{pifont}
\usepackage{microtype}      
\usepackage{graphicx}
\usepackage{doi}
\usepackage{authblk}
\usepackage{tikz}
\usepackage{tablefootnote}

\usepackage{amsmath,amssymb,amsthm}
\usepackage{mathrsfs} 
\usepackage{euscript}
\usepackage{charter}
\usepackage{array}
\usepackage[normalem]{ulem}

\definecolor{labelkey}{rgb}{0,0.08,0.45}
\definecolor{refkey}{rgb}{0,0.6,0.0}
\definecolor{Brown}{rgb}{0.45,0.0,0.05}
\definecolor{dgreen}{rgb}{0.00,0.49,0.00}
\definecolor{dblue}{rgb}{0,0.08,0.75}

\hypersetup{
colorlinks,
linkcolor=dgreen,
citecolor=dblue
}

\newcommand{\iter}{i}
\newcommand{\Iter}{k}
\newcommand{\scO}{\mathcal{O}}

\newcommand{\Y}{\mathcal{Y}}
\newcommand{\cZ}{\mathcal{Z}}

\newcommand{\Ex}{\mathbb{E}}

\newcommand{\N}{\mathbb{N}}

\newcommand{\R}{\mathbb{R}}
\newcommand{\Prob}{\mathbb{P}}

\DeclareMathOperator*{\argmin}{argmin}

\newtheorem{lemma}{Lemma}

\newtheorem{fact}{Fact}
\newtheorem{corollary}{Corollary}
\newtheorem{theorem}{Theorem}

\theoremstyle{definition}

\newtheorem{remark}{Remark}

\newcommand{\abs}[1]{\lvert{#1}\rvert}
\newcommand{\norm}[1]{\lVert{#1}\rVert}

\newcommand{\scalarp}[1]{\langle{#1}\rangle}
\newcommand{\minimize}[2]{\ensuremath{\underset{\substack{{#1}}}%
{\text{\rm minimize}}\;\;#2 }}

\newcommand{\HH}{H}
\newcommand{\EE}{\ensuremath{\mathsf E}}

\newcommand*\circled[1]{\tikz[baseline=(char.base)]{
            \node[shape=circle,draw,inner sep=2pt,font=\scriptsize] (char) {#1};}}

\numberwithin{equation}{section}

\title{ { \sffamily High Probability Bounds for Stochastic Subgradient Schemes with Heavy Tailed Noise
} }

\author[1, 5]{Daniela A. Parletta}
\author[2]{Andrea Paudice}
\author[3, 5]{Massimiliano Pontil}
\author[4, 5]{Saverio Salzo}

\affil[1]{\footnotesize Department of Mathematics, University of Genoa, Via Dodecaneso 35, 16146 Genova, Italy}
\affil[2]{\footnotesize Department of Computer Science, University of Milan, Via Giovanni Celoria 18, 20133 Milano, Italy}
\affil[3]{\footnotesize Department of Computer Science, UCL, Malet Place, London WC1E 6BT, UK}
\affil[4]{\footnotesize DIAG, Sapienza University of Rome, Via Ariosto, 25, 00185 Roma, Italy}
\affil[5]{\footnotesize CSML, Istituto Italiano di Tecnologia, Via Enrico Melen 83, 16152 Genova, Italy}

\date{}

\begin{document}
\maketitle

\begin{abstract}
In this work we study high probability bounds for stochastic subgradient methods under heavy tailed noise in Hilbert spaces. In this setting the noise is only assumed to have finite variance as opposed to a sub-Gaussian distribution for which it is known that standard subgradient methods enjoy high probability bounds. We analyzed a clipped version of the projected stochastic subgradient method, where subgradient estimates are truncated whenever they have large norms. We show that this clipping strategy leads both to optimal \emph{anytime} and finite horizon bounds for general averaging schemes of the iterates. We also show an application of our proposal to the case of kernel methods which gives an efficient and fully implementable algorithm for statistical supervised learning problems. Preliminary experiments are shown to support the validity of the method.
\end{abstract}

 \vspace{1ex}

\allowdisplaybreaks

\section{Introduction}
The subgradient method was introduced in the 1960s by the Russian school of optimization as the natural generalization of the gradient descent method to nonsmooth functions. It was first devised by Shor in 1962, and later studied by Polyak, Demyanov, and Ermoliev  \cite{Polyak1967,Polyak1977,Shor1985, Bertsekas1999}. The stochastic version of this algorithm was considered by Ermoliev in 1969 \cite{Ermoliev1969}, who focused on the convergence of the iterates. Later on, such method has been extensively studied, with most existing results providing upper bounds on the expected optimization error in function values.
Indeed, state of the art convergence results ensure the optimal rate of $\scO\big(1/\sqrt{k}\big)$ for convex Lipschitz functions \cite{Nemirovskij1983,Shamir2013}. On the other hand, high probability bounds have been proved harder to obtain. Differently from bounds in expectation, most high-probability bounds have been derived under {\it light tails} assumptions, meaning with sub-Gaussian noise \cite{Harvey2019a,Harvey2019b,Jain2019}.
Recently, motivated by the fact that real world datasets are abundant but of poor quality, a line of research has started investigating high-probability bounds with {\it heavy-tails} assumptions, that is, with uniformly bounded variance noise.
High probability bounds in this setting, have been proved in \cite{Nazin2019,Gorbunov2021}. Despite obtaining near-optimal rates, both works suffer from either unpractical parameter settings or unrealistic assumptions. Moreover, differently from most results obtained in the light tailed case, in
\cite{Nazin2019,Gorbunov2021}
the analysis is confined to a finite horizon, which is a limitation in many practical scenarios. 
Indeed, finite horizon methods cannot cope with online settings
in which data arrives continuously in a potentially infinite stream of batches and the predictive model is updated accordingly.

In this work we address the following optimization problem
\begin{equation}
\label{eq:mainprob}
\minimize{x \in X}{f(x)},
\end{equation}
where $X\subset \HH$ is a nonempty closed convex and bounded set in a Hilbert space $H$ with diameter $D\geq0$ and $f\colon \HH \to \R$ is a convex Lipschitz continuous function with Lipschitz constant $L>0$. We assume that the projection onto $X$ can be computed explicitly but only a stochastic subgradient of $f$ is available, that is, that for all $x \in X$, we have
\begin{enumerate}
\item $\hat{u}(x,\xi) \in H$ and $\xi$ is a random variable such that $\Ex[\hat{u}(x,\xi)] \in \partial f(x)$.
\item $\Ex[\norm{\hat{u}(x,\xi) - \Ex[\hat{u}(x,\xi)]}^2] \leq \sigma^2$.
\end{enumerate}
We stress that the only assumption made on the stochastic subgradient is that it has uniformly bounded variance, while no additional information on its distribution is available. 
\paragraph{Contributions.}
In relation to problem~\eqref{eq:mainprob}, we study a {\it projected clipped stochastic subgradient method} for which we provide high-probability convergence rates under heavy tailed noise. 
The main contributions of this work are as follows.
\begin{itemize}
\item 
We made a new analysis of the clipped subgradient method which relies on a new decomposition of the error and different statistical properties which contrasts our analysis with that of \cite{Gorbunov2020,Gorbunov2021,Nazin2019,Holland2022}. This allows to obtain a first bound on the objective values which is valid for arbitrary stepsizes, weights, and clipping levels, making their role more transparent in the study of the convergence. See Theorem~\ref{thm:main} and Remark~\ref{rmk:main1}.
\item 
We provide a general convergence result when the algorithm parameters 
obey polynomial laws. This makes clear 
the setting of the stepsizes, weights and clipping levels so to have the optimal convergence rate of $\mathcal{O}(1/\sqrt{k})$.
See Theorem~\ref{thm:main2}.
\item The analysis covers both the finite-horizon and the infinite-horizon settings, general averaging schemes, and can reveal subGaussian tail behavior $\mathcal{O}(\sqrt{\log(\delta^{-1})}/\sqrt{k})$ of the function values. See Corollary~\ref{cor:2023829}.
\item We provide an application of the proposed method to the important case of statistical learning with kernels. Notably, the resulting algorithm is fully practicable and achieves the optimal $\scO\big(1/\sqrt{k}\big)$ rate of convergence in high probability for the excess risk.
\end{itemize}

\subsection{Related Work}

In this section we discuss the current literature on high-probability bounds for the stochastic non-smooth convex setting. A summary of the state-of-the-art is given in \Cref{t:RLcmpx}. 

\paragraph{Light Tails} 
Convergence rates in high probability for light tails noise have been derived in \cite{Harvey2019a,Harvey2019b,Jain2019}. In particular, in \cite{Harvey2019a}, under sub-Gaussian hypothesis, the last iterate of SGD was shown to achieve a convergence rate of $\scO\big((\log k/\sqrt{k})\log(1/\delta) \big)$, even with infinite horizon. The authors also state, without proof, that the average iterate obtains the improved rate of $\scO\big(\log(1/\delta)/\sqrt{k}\big)$, essentially matching the analogue bound in expectation. For strongly convex functions, they show that the last iterate achieves the rate $\scO\big((\log k/k)\log(1/\delta) \big)$, while the {\it suffix average} improves to $\scO\big(1/k\log(1/\delta)\big)$. Unfortunately, the {\it suffix average} can be tricky to implement in an infinite-horizon setting. Therefore, in \cite{Harvey2019b} a simpler-to-implement weighted averaging scheme is shown to obtain the same rate. In \cite{Jain2019} the previous results on the last iterate are improved to the optimal rate, but only when the time horizon is known in advance and the noise is bounded almost surely.

\paragraph{Heavy Tails}
High-probability bounds in this context have been derived for some special settings in \cite{Davis2021,Nazin2019}. The work \cite{Davis2021} considers a class of nonsmooth composite functions, where the objective is the sum of a smooth strongly convex function and a general closed convex term. The authors propose an elegant and neat method, \textsc{ProxBoost}, that combines a robust statistical estimation procedure with the proximal point method to boost any bound in expectation into an high-probability guarantee. When used to boost the optimal method of \cite{Ghadimi2013}, \textsc{ProxBoost} achieves the optimal rate of  $\scO((1/k)\sqrt{\log\left(1/\delta\right)})$. Unfortunately, this method has some practical shortcomings. First, the stochastic oracle is supposed to be available only for the smooth term of the objective function, which, in addition, is essentially limited to quadratics. Second, the proposed method requires at least three nested loops, the outer of which being the proximal point algorithm. This leads to an overall procedure which is rather convoluted and rise concerns about its practicability.

The previous composite objective structure is also analyzed in \cite{Nazin2019}. The authors remove 
the strong convexity assumption on the smooth part, but requires the minimization to be performed over a convex bounded domain.
They develop \textsc{RSMD}, a robust version of stochastic mirror descent enjoying several desirable properties such as a near-optimal rate of $\scO((1/\sqrt{k})\log(1/\delta))$, the absence of a batch-size and a simple constant step-size.
On the other hand, the algorithm relies on the following gradient estimator
\begin{equation*}
\tilde{u}_k \coloneqq 
\begin{cases}
\hat{u}(x_k, \chi^k), \text{ if } \|\hat{u}(x_k, \xi^k) - \bar{g} \| \leq L \|x_k-\bar{x}\| + \lambda + \nu \sigma, \\[1ex]
\bar{g}, \text{ otherwise.}
\end{cases}
\label{eq:truncatedGradient}
\end{equation*}
where $\bar{g}$ is a gradient estimate which is {\it close enough} to the true gradient of the smooth part of the objective at a given point $\bar{x}$. This estimator is problematic, since it is not clear how to obtain $(\bar{g}, \bar{x})$ in practice. The authors suggest to generate (with high-probability) $\bar{g}$ by first sampling {\it enough} oracle estimates at $\bar{x}$ and then by computing their geometric median; a fact that reduce the overall confidence of the procedure. However, computing such mean estimator in high dimensions is a challenging task. Furthermore, the truncation parameter $\lambda$ has to be set to 
\begin{equation*}
\lambda = \sigma \sqrt{\frac{k}{\log\left(1/\delta\right)}} + \nu\sigma,
\label{eq:truncatedGradientparam}
\end{equation*}
which requires the knowledge of the time horizon $k$ and does not allow for any-time guarantees. A follow-up study is \cite{Holland2022}, where the author modifies the any-time to batch conversion from \cite{Cutkosky2019} 
so to extend the guarantees of \textsc{RSMD} 
from the last average iterate, to all the averages generated till the given time horizon. We note that the resulting algorithm, is not fully any-time, since the initial truncation level needs to be set according to the time horizon. Differently from \cite{Nazin2019}, it is assumed that the objective is smooth, while the non-smooth case is left as an open problem. Moreover, since the gradient estimator is the same as in \cite{Nazin2019}, it suffers from the same practicability issues we discussed above and the author leaves as an open problem that of identifying a more practical estimator.
 
Clipping strategies have already been used in \cite{Gorbunov2020,Gorbunov2021}, which provide convergence rates in high probability for a fixed-horizon setting. In contrast to our work, they cover the {\it unconstrained} minimization of convex Lipschitz smooth functions and convex Lipschitz continuous functions respectively. The authors analyze a clipped version of SGD, but the resulting bounds and the algorithm's parameters are subject to
some limitations. In particular, the work \cite{Gorbunov2021}, which addresses our setting, requires the algorithm parameters to be set as follows.

\begin{table}[t]
\centering 
\resizebox{\textwidth}{!}{%
\begin{tabular}{ccccccc}
Method  & Rate & Noise & Any-Time & Function Type & Constraints & Ref. \\ 
\midrule
\hline
& & & & & \\ 
\textsc{SGD} & {\small $\frac{\log k}{\sqrt{k}} \log \left(1/\delta\right)$} & LT & \ding{51} & Lipschitz & \ding{51} (Bounded) & \cite{Harvey2019a} \\
& & & & & \\ 
\hline
& & & & & \\ 
\textsc{SGD} & {\small $\sqrt{\frac{\log \left(1/\delta\right)}{k}}$} & LT & \ding{51}\tablefootnote{The rate anytime on the average iterate follows easily from \cite[Proposizion~4.1]{Lan2020}. For readers' convenience we derive this rate in Appendix~\ref{appC}} & Lipschitz & \ding{51} (Bounded) & \cite{Jain2019} \\ 
& & & & & \\
\hline
& & & & & \\ 
\textsc{RSMD} & {\small $\sqrt{\frac{\log(1/\delta)}{k}}$} & HT & \ding{55} & Composite & \ding{51} (Bounded) & \cite{Nazin2019} \\ 
& & & & & \\
\hline
& & & & & \\ 
\textsc{ClippedSGD} & {\small $\sqrt{\frac{\log\left(k/\delta\right)}{k}}$} & HT & \ding{55} & Lipschitz & \ding{55} & \cite{Gorbunov2021} \\ 
& & & & & \\
\hline
\rowcolor{green!10}
& & & & & & \\ 
\rowcolor{green!10}
\textsc{ClippedSGD} & {\small $\sqrt{\frac{\log(1/\delta)}{k}}$} & HT & \ding{51} & Lipschitz & \ding{51} (Bounded) & This work \\ 
\rowcolor{green!10}
& & & & & & \\
\hline
\end{tabular}}
\vspace{2ex}
\caption{Comparison of known high-probability convergence rate for nonsmooth problems. ``Noise'' refers to the type of oracle noise with ``LT'' and ``HT'' standing for Light Tails and Heavy Tails respectively. ``Any-Time'' refers to any-time convergence guarantees. ``Function Type'' refers to the smoothness assumptions on $f$, with ``Lipschitz'' denoting to the case of a convex Lipschitz objective, and ``Composite'' referring to a composite objective $f+h$ where $f$ is convex and smooth and $h$ is a given (non stochastic) convex function. ``Constraints'' refers to the constrain set, ``Bounded'' refers to the optimization over a convex bounded subset of an Hilbert space.}
\label{t:RLcmpx}
\end{table}

\begin{align}
\label{eq:20230830c}
\gamma \leq \min\left\{\frac{\varepsilon}{8 L^2}, \frac{D}{\sqrt{2k}L}, \frac{D}{2L \log\left(4k/\delta\right)}\right\} & \quad \text{Step-Size} \\
\label{eq:20230831a1} 
m \geq \max \left\{ 1, \frac{81 k\sigma^2}{\lambda^2 \log\left(4k/\delta\right)} \right\} & \quad \text{Batch-Size} \\ 
\label{eq:20230831a}
\lambda = \frac{D}{\gamma \log\left(4k/\delta\right)} & \quad \text{Clipping-Level},
\end{align}
where $k$ is the fixed-horizon and $\varepsilon>0$ a free parameter. We note the restrictions in the range of the step-size $\gamma$ and the batch-size $m$. In addition, they are coupled together. Indeed, in \cite[Corollary 5.1]{Gorbunov2021} they show that it is possible to use $m=1$ but at the expense of shrinking the step-size $\gamma$ to $\scO\big(1/\sqrt{k\log(k)}\big)$, so that the final rate reduces to the order of $\scO\big(\sqrt{\log(k/\delta)/k}\big)$, which is not optimal in this setting. Moreover, all the parameters depend on the time horizon $k$, so that the convergence guarantees are not any-time. On the other hand \cite{Gorbunov2020,Gorbunov2021} consider the minimization on unbounded domains, which is possibly more challenging than our setting.

Finally we note that stochastic gradient methods have been also
studied in conjunction with biased compressor (nonlinear) operators. See, e.g., \cite{Richtarik2021} and reference therein. In this respect we note that the clipping operator, being a projection onto a ball, is not a compressor and moreover it is invoked dynamically with time-varying radii.

\subsection{Notation and basic facts}
We set $\N = \{1,2,\dots, \}$ the set of natural numbers starting from $1$.
We set $\R_+ = \left[0,+\infty\right[$ and $\R_{++}=\left]0,+\infty\right[$ the 
sets of positive and strictly positive real numbers respectively. 
For any $p,q \in \R$ we set $p\vee q = \max\{p,q\}$, $p\wedge q = \min\{p,q\}$, and $p_+ = p\vee 0$.
We denote by $\log$ the natural logarithm function. 
A sequence of real numbers $(\chi_k)_{k \in \N}$ is called increasing if, for every $k \in \N$,
$\chi_k \leq \chi_{k+1}$.
In the following $H$ will 
be a real Hilbert space and $\scalarp{\cdot, \cdot}$ and $\norm{\cdot}$
will denote its scalar product and associated norm. For a convex function $f\colon H \to \R$, the subdifferential of $f$
at a point $x \in \HH$ is defined as
\begin{equation*}
    \partial f(x) = \big\{ u \in \HH\,\big\vert\,\forall\, y \in \HH\colon f(y) \geq f(x) + \scalarp{y-x,u} \big\}.
\end{equation*}
For a closed convex set $X \subset H$ we denote by $P_X$ the orthogonal projection operator onto $X$.

We recall the following Bernstein's inequality for martingales \cite{Freedman1975}.

\begin{fact}[Freedman's Inequality for Martingales]
\label{thm:freedman}
Let $(X_k)_{k \in \N}$ be a martingale difference sequence such that for all $k$
\begin{enumerate}[leftmargin=.75cm, label={\rm (\roman*)}]
\item $|X_k| \leq c$ a.s.,
\item $\sigma_k^2 \coloneqq \Ex[X^2_k |X_1,\dots,X_{k-1}] < \infty$.
\end{enumerate}
Let $k \in \N$ and set $V_k = \sum_{\iter=1}^k \sigma_i^2$. Then for every $\eta$ and $F$ in $\R_{++}$,
\begin{equation}
\Prob \left(\sum_{\iter=1}^n X_i > \eta, V_k \leq F \right) \leq \exp \left(- \frac{1}{2} \frac{\eta^2}{F + \frac{1}{3}\eta c}\right) \;.
\end{equation}
\end{fact}

\section{The Algorithm and the Main Results}

In this section we detail the algorithm and describe the main results of the paper.

Since we are in a heavy-tails regime we consider to clip the stochastic subgradient oracle to a given level. We therefore define the following clipping operation, which corresponds to a projection onto the closed ball with radius $\lambda$,
\begin{equation*}
(\forall\, u \in \HH)(\forall\, \lambda \in \R_{++})\qquad
\textsc{CLIP}(u, \lambda) = \min\bigg\{\frac{\lambda}{\norm{u}}, 1\bigg\} u = 
\begin{cases}
u &\text{if } \norm{u}\leq \lambda\\
\frac{\lambda}{\norm{u}}u  &\text{if } \norm{u}>\lambda.
\end{cases}.
\end{equation*}
The algorithm is detailed below.
\begin{algorithm}
\caption{Clipped Stochastic subGradient Method (C-SsGM)}
Given the stepsizes $(\gamma_k)_{k \in \N} \in \R_{++}^\N$, the weights $(w_k)_{k \in \N} \in \R_{++}^\N$, the clipping levels $(\lambda_k)_{k \in \N} \in \R_{++}^\N$, the batch size $m \in \N$, $m\geq 1$, and an initial point $x_1 \in X$,
\begin{equation}
\label{eq:algo}
\begin{array}{l}
\text{for}\;k=1,\ldots\\[1ex]
\left\lfloor
\begin{array}{l}
\text{draw } \bm{\xi}^k = (\xi_j^k)_{1 \leq j \leq m}\ \ m \text{ independent copies of } \xi\\[1ex]
\bar{u}_k = \displaystyle\frac 1 m \sum_{j=1}^k \hat{u}(x_k, \xi^k_j),\\[3ex] 
\tilde{u}_k = 
\textsc{CLIP}(\bar{u}_k, \lambda_k)\\[1ex]
x_{k+1} = P_X(x_k-\gamma_k \tilde{u}_k).
\end{array}
\right.
\end{array}
\end{equation}
From the sequence $(x_k)_{k \in \N}$ one defines also
\begin{equation}
(\forall\, k \in \N)\qquad\bar{x}_k = \bigg( \sum_{\iter=1}^k w_i\bigg)^{-1} \sum_{\iter=1}^k w_i x_i.
\end{equation}
\label{algo:projsub2}
\end{algorithm}
\vspace{1ex}
\begin{remark}
In addition to the sequence $x_k$, Algorithm~\ref{algo:projsub2} requires keeping track of the sequences  $W_k := \sum_{\iter=1}^k w_i$ and $\bar{x}_k$, which can be updated recursively, as $W_{k+1} = W_k + w_{k+1}$ and $\bar{x}_{k+1} = W^{-1}_{k+1} (W_k \bar{x}_k + w_{k+1} x_{k+1})$.
\end{remark}

We will establish high probability convergence rates in several situations depending on the choices of the weights $w_k$'s, the stepsizes $\gamma_k$'s, and the clipping levels $\lambda_k$'s. 

Now we are ready to provide the first main result of the paper.

\begin{theorem}[Main result 1]
\label{thm:main}
Suppose that, for every $k \in \N$, $\lambda_k \geq (1+\varepsilon) L$ with $\varepsilon > 0$ and that the sequence $(w_k/\gamma_k)_{k \in \N}$ is increasing. Let $(\bar{x}_k)_{k \in \N}$ be the sequence generated by Algorithm~\ref{algo:projsub2}. Then for every $k \in \N$ and $\delta \in \left]0,2/e\right]$, the following holds with probability at least $1-\delta$
\begin{align*}
f(\bar{x}_k) - \min_X f &\leq 
\frac{1}{\sum_{\iter=1}^k w_i}\bigg[
\frac{D^2}{2}
\frac{w_k}{\gamma_k} 
+ \frac{2}{3} \left(2 D \cdot \max_{1 \leq \iter \leq \Iter} w_\iter \lambda_\iter + \max_{1 \leq \iter \leq \Iter} w_\iter \gamma_\iter \lambda^2_\iter \right) \cdot \log\left(\frac{2}{\delta}\right) \\
&\qquad+ \frac{1}{\sqrt{2}}\left(4 \bigg(1 + \frac 1 \varepsilon \bigg) \frac{D \sigma}{\sqrt{m}} \sqrt{\sum_{\iter=1}^\Iter 
w_\iter^2} + \sqrt{\bigg(\frac{\sigma^2}{m} + L^2 \bigg) \sum_{\iter=1}^\Iter w_\iter^2 \gamma_\iter^2 \lambda_\iter^2} \right)\sqrt{\log\left(\frac{2}{\delta}\right)} \\ 
&\qquad + \frac{D \sigma^2}{m} \bigg(1+ \frac 1 \varepsilon \bigg) \sum_{\iter=1}^\Iter \frac{w_\iter}{\lambda_\iter}
+ \frac{1}{2} \bigg(\frac{\sigma^2}{m} + L^2\bigg) \sum_{\iter=1}^\Iter w_\iter \gamma_\iter
\bigg].
\end{align*}
\end{theorem}
\begin{remark}
\label{rmk:main1}
Due the above result, it is clear that in order to ensure convergence of Algorithm~\ref{algo:projsub2} we need to control the following quantities
\begin{align}
\label{20230822a}
&\underbrace{\frac{w_k/\gamma_k}{\sum_{i=1}^k w_\iter}}_{
\circled{1}
},\quad
\underbrace{\frac{\max_{1 \leq i \leq k} w_i \lambda_\iter}{\sum_{i=1}^k w_i}}_{\circled{2}},\quad
\underbrace{\frac{\max_{1 \leq i \leq k} w_i \gamma_i \lambda_\iter^2}{\sum_{i=1}^k w_i}}_{\circled{3}}\\[1ex]
\label{20230822b}
&\underbrace{\frac{\sqrt{\sum_{i=1}^k w^2_i}}{\sum_{i=1}^k w_i}}_{\circled{4}}, \quad 
\underbrace{\frac{\sqrt{\sum_{i=1}^k w^2_i \gamma_i^2 \lambda_i^2}}{\sum_{i=1}^k w_i}}_{\circled{5}},\quad \underbrace{\frac{\sum_{i=1}^k w_\iter/\lambda_\iter}{\sum_{i=1}^k w_i}}_{\circled{6}}, \quad \underbrace{\frac{\sum_{i=1}^k w_i \gamma_i}{\sum_{i=1}^k w_i}}_{\circled{7}}.
\end{align}
With the exception of term 1 which is standard in the analysis of SsGM, the rest of the quantities are related to the bias (2,4,6)
and the variance (3,5,7) of the subgradient estimator.
\end{remark}

In the rest of the section we will assume that
the constants $w_k, \lambda_k, \gamma_k$ are set as follows
\begin{equation}
\label{eq:params}
    w_k = k^p,\quad \gamma_k = \frac{\gamma}{k^r}, \quad
    \lambda_k = \max\{\beta k^q, (1+\varepsilon) L\}\quad(p,r,q \in \R, \gamma, \beta, \varepsilon>0)
\end{equation}
and we will show conditions on the exponents so to make the quantities in \eqref{20230822a} and \eqref{20230822b} converging to zero. The related result is given below.
\begin{theorem}[Main result 2]
\label{thm:main2}
Let $(w_k)_{k \in \N}$, $(\gamma_k)_{k \in \N}$ and $(\lambda_k)_{k \in \N}$ be defined as in \eqref{eq:params}
with $p,r,q \in \R$ and $\beta, \varepsilon>0$. Let $\delta \in \left]0,2/e\right]$. Then Algorithm~\ref{algo:projsub2} converges in high probability provided that the following conditions are satisfied
\begin{equation}
\label{eq:convergencecond}
    p>-r,\ r \in \left]0,1\right[,\ q \in \left]0,1\right[,\ \text{and}\ q<\frac{r+1}{2}
\end{equation}
and in such case, with probability greater then $1-\delta$, the following (infinite-horizon) rate holds 
\begin{multline*}
f(\bar{x}_k)- \min_{X} f =\mathcal{O} \bigg(
\frac{1}{k^{1-r}}
+\frac{1}{k^{\min\{p+1, 1-q\}}}
+ \frac{1}{k^{\min\{p+1, r+1-2q\}}}\\
+\frac{1}{k^{\min\{p+1, \frac 1 2 \}}}
+\frac{1}{k^{\min\{p+1, r+ \frac 1 2 - q\}}}
+ \frac{1}{k^{\min\{p+1, q\}}}
+\frac{1}{k^{\min\{p+1, r\}}} 
\bigg).
\end{multline*}
Moreover, the optimal choices, in terms of convergence rates, for the parameters $r$ and $q$ are $r=1/2$ and $q=1/2$, and in such case the following hold.
\begin{enumerate}[leftmargin=.75cm,label={\rm (\roman*)}]
\item\label{thm:main2_i} Suppose that $p>-1/2$ and set $L_\varepsilon = (1+\varepsilon) L$. Then for every $k \in \N$, with probability at least $1-\delta$, we have
\begin{align*}
f(\bar{x}_k)- \min_X f &\leq \frac{(p+1)\vee 1}{\sqrt{k}} 
\bigg[ \frac{D^2}{2 \gamma} \\[1ex]
&\quad + \frac 2 3 \bigg( 2 D \max\Big\{\beta, \frac{L_\varepsilon}{k^{\min\{p+ \frac 1 2, \frac 1 2\}}}\Big\} + \gamma \max\Big\{ \beta^2, \frac{L_\varepsilon^2}{k^{\min\{p+\frac 1 2, 1\}}} \Big\} \bigg) 
\log \frac{2}{\delta}\\[2ex]
&\quad + \frac{1}{\sqrt{2}} \bigg( 4 \left(1+ \frac 1 \varepsilon\right) 
\frac{D \sigma}{\sqrt{m}} \frac{1}{((2p+1)\wedge 1)} + \gamma v_p(k)\sqrt{\frac{\sigma^2}{m}+L^2} \bigg) \sqrt{\log \frac{2}{\delta}}\\[2ex]
&\quad +\frac{1}{(p+ 1/2)\wedge 1} \bigg( \frac{D \sigma^2}{m} \left( 1 + \frac 1 \varepsilon \right) \frac 1 \beta +\frac{\gamma}{2} \left( \frac{\sigma^2}{m} + L^2\right) \bigg)
\bigg],
\end{align*}
where
\begin{equation*}
v_p(k) = 
\begin{cases}
\displaystyle
\sqrt{\frac{2}{(2p)\wedge 1}} \max\Big\{\beta, \frac{L_\varepsilon}{\sqrt{k}} \Big\} &\text{if } p>0 \\[3ex]
\begin{cases}
L_\varepsilon &\text{if } k< L_\varepsilon^2/\beta^2\\[2ex]
\displaystyle
\beta \left( \frac{L_\varepsilon^2/\beta^2}{\sqrt{k}}+1\right) 
&\text{if } k\geq L_\varepsilon^2/\beta^2
\end{cases}&\text{if } p=0 \\[6ex]
\displaystyle
\frac{L_\varepsilon}{k^{p+\frac1 2}} \sqrt{1 - \frac{1}{2p}} + \frac{\max\{ \beta, L_\varepsilon/\sqrt{k}\}}{\sqrt{(2p+1)\wedge 1}} &\text{if } -\frac 1 2 <p<0
\end{cases}
\end{equation*}
\item\label{thm:main2_ii} Suppose that $p=-1/2$ and set $L_\varepsilon = (1+\varepsilon) L$. Then for every $k \in \N$, with probability at least $1-\delta$, we have
\begin{align*}
f(\bar{x}_k)- \min_X f &\leq \frac 3 2 \frac{1}{\sqrt{k}} 
\bigg[ \frac{D^2}{2 \gamma} \\[1ex]
&\quad + \frac 2 3 \bigg( 2 D \max\big\{\beta, L_\varepsilon \big\} + \gamma \max\big\{ \beta^2, L_\varepsilon^2 \big\} \bigg) 
\log \frac{2}{\delta}\\[2ex]
&\quad + \frac{1}{\sqrt{2}} \bigg( 4 \left(1+ \frac 1 \varepsilon\right) 
\frac{D \sigma}{\sqrt{m}} (1+ \log k) + \gamma u(k) \sqrt{\frac{\sigma^2}{m}+L^2} \bigg) \sqrt{\log \frac{2}{\delta}}\\[2ex]
&\quad + \bigg( \frac{D \sigma^2}{m} \left( 1 + \frac 1 \varepsilon \right)
\frac 1 \beta +\frac{\gamma}{2} \left( \frac{\sigma^2}{m} + L^2\right) \bigg) (1+\log k)
\bigg],
\end{align*}
where $u(k) = \sqrt{2}L_\varepsilon + \beta (1+\log k)$.
\item\label{thm:main2_iii} (Finite-horizon) Let $k \in \N$ and set $(\gamma_i)_{1 \leq i \leq k} \equiv \gamma/\sqrt{k}$, $(\lambda_i)_{1 \leq i \leq k} \equiv (\max\{\beta \sqrt{i}, (1+\varepsilon) L\})_{1 \leq i \leq k}$, and
$(w_i)_{1 \leq i \leq k} \equiv (i^p)_{1 \leq i \leq k}$ with $p \geq 0$. Then with probability at least $1-\delta$, we have
\begin{align*}
f(\bar{x}_k)- \min_X f &\leq \frac{p+1}{\sqrt{k}} 
\bigg[ \frac{D^2}{2 \gamma} \\[1ex]
&\quad + \frac 2 3 \bigg( 2 D \max\Big\{\beta, \frac{L_\varepsilon}{\sqrt{k}}\Big\} + \gamma \max\Big\{ \beta^2, \frac{L_\varepsilon^2}{k} \Big\} \bigg) 
\log \frac{2}{\delta}\\[2ex]
&\quad + \frac{1}{\sqrt{2}} \bigg( 4 \left(1+ \frac 1 \varepsilon\right) 
\frac{D\sigma}{\sqrt{m}} + \gamma y_p(k)\sqrt{\frac{\sigma^2}{m}+L^2}  \bigg) \sqrt{\log \frac{2}{\delta}}\\[2ex]
&\quad +\frac{1}{(p+ 1/2)\wedge 1} \frac{D \sigma^2}{m} \left( 1 + \frac 1 \varepsilon \right) \frac 1 \beta + \frac{\gamma}{2} \left( \frac{\sigma^2}{m} + L^2\right) 
\bigg],
\end{align*}
where $y_p(k) = \sqrt{2} \max\{\beta, L_\varepsilon/\sqrt{k}\}$.
\end{enumerate}
\end{theorem}
\begin{remark}\label{rmk:20240413a}\ 
\begin{enumerate}[leftmargin=.75cm,label={\rm (\roman*)}]
\item The previous theorem provides always convergence (not necessarily optimal) for general parameter settings.
In literature it is common to assume from the beginning that $r=q=1/2$ \cite{Nazin2019, Gorbunov2021,Holland2022}. In contrast our analysis
shows {\it a posteriori} that those choices are optimal.
\item\label{rmk:20240413a_ii}  The previous theorem shows that for $p>-1/2$ we have an asymptotic (any time) convergence rate of $\mathcal{O}(1/\sqrt{k})$ (note that $v_p(k)$ is bounded from above) and that for $p=-1/2$ the rate degrades to $(1+\log k)/\sqrt{k}$.
\item Point \ref{thm:main2_iii} of Theorem~\ref{thm:main2} provides rate in the finite horizon setting in which $k$ (the time horizon) is fixed {\it a priori} and the stepsize is constant, set according to that time horizon. Moreover, the form of the bound allows to optimize the stepsize. Indeed, if $k \geq (L_\varepsilon/\beta)^2$, the best stepsize (minimizing the bound) is
\begin{equation*}
\frac{\gamma}{\sqrt{k}} = 
\frac{D}{\sqrt{k}} \bigg( \frac 4 3 \beta^2 \log \frac 2 \delta + 2 \sqrt{\frac{\sigma^2}{m} + L^2} \sqrt{\log \frac{2}{\delta}} + \Big( \frac{\sigma^2}{m} + L^2 \Big) \bigg)^{-1/2}.
\end{equation*}
Compared to \eqref{eq:20230830c}--\eqref{eq:20230831a} (obtained in \cite{Gorbunov2020,Gorbunov2021}),
we see that in our case the batchsize $m$ and the parameter $\beta$ of the clipping levels are completely free and the rate of $\mathcal{O}(1/\sqrt{k})$ is always guaranteed. Note also that here the clipping level is not constant within the time horizon $k$, but it follows the policy $\lambda_i = \max\{\beta \sqrt{i}, (1+\varepsilon)L\}$.
\end{enumerate}
\end{remark}
\vspace{1ex}

From Theorem~\ref{thm:main2}\ref{thm:main2_i}-\ref{thm:main2_iii} it is easy to derive the following result which essentially removes the $\log$ factor in the bound.
\begin{corollary}
\label{cor:2023829}
Let $\delta \in \left]0,2/e\right]$ and set $\beta = \bar{\beta}/\sqrt{\log (2/\delta)}$, with $\bar{\beta} >0$.
Let $(w_k)_{k \in \N}$, $(\gamma_k)_{k \in \N}$ and $(\lambda_k)_{k \in \N}$ be defined as in \eqref{eq:params}
with $q=r=1/2$, $p>-1/2$ and $\varepsilon>0$. Let $(\bar{x}_k)_{k \in \N}$ be the sequence generated by Algorithm~\ref{algo:projsub2}. Then, for every $k \geq (L_\varepsilon/\beta)^{\max\{\frac{2}{2p+1},2\}}$, if $p\neq 0$, and every $k \geq (L_\varepsilon/\beta)^4$, if $p=0$, and with probability at least $1-\delta$, we have
\begin{align*}
f(\bar{x}_k)- \min_X f &\leq \frac{(p+1)\vee 1}{\sqrt{k}} 
\bigg[ \frac{D^2}{2 \gamma}
 + \gamma\bigg( \frac 2 3 \bar{\beta}^2 + \frac{a_p\bar{\beta}}{\sqrt{2}} \sqrt{\frac{\sigma^2}{m}+L^2}  + \frac{\sigma^2/m + L^2}{2((p+ 1/2)\wedge 1)}  \bigg) \\[2ex]
&\quad + D\bigg(\frac 4 3 \bar{\beta}  +  \left(1+ \frac 1 \varepsilon\right)  \bigg( 
 \frac{1}{(2p+ 1)\wedge 1}\frac{4 \sigma}{\sqrt{2 m}} +
 \frac{1}{(p+ 1/2)\wedge 1}\frac{\sigma^2}{m\bar{\beta}}  \bigg) \bigg)\sqrt{\log \frac 2 \delta} 
\bigg],
\end{align*}
where
\begin{equation*}
a_p=
\begin{cases}
\sqrt{2 ((2p)\wedge 1)^{-1}} &\text{if } p>0\\[1ex]
2 &\text{if } p=0\\
\sqrt{1 - (2p)^{-1}} + \sqrt{((2p+1)\wedge 1)^{-1}}&\text{if } -1/2<p<0.
\end{cases}
\end{equation*}
Moreover, for finite horizon setting, with $(\gamma_{i})_{1\leq i \leq k} \equiv \gamma/\sqrt{k}$, $q=1/2$, and $p\geq 0$, we have exactly the same bound except for $k$ that should be taken larger than $(L_\varepsilon/\beta)^2$, the constant $(p+1/2)\wedge 1$ at the denominator of the last term in the first line which is replaced by $1$, and the constant $a_p$ which now is equal to $\sqrt{2}$.
\end{corollary}

\newpage
\begin{remark}\label{rmk:20230830}\ 
\begin{enumerate}[leftmargin=.75cm, label={\rm (\roman*)}]
\item\label{rmk:20230830_i} Corollary~\ref{cor:2023829} shows that for $k$ large enough the variable $f(\bar{x}_k)- \min_X f$
shows a subGaussian tail behavior. In this respect we note that with $p \geq 1/2$ this behavior occurs earlier (with $k \geq (L_\varepsilon/\beta)^2$). This same condition on $k$ ($k \gtrsim \log \delta^{-1}$) occurs also in other works \cite{Nazin2019,Gorbunov2021,Holland2022}.
\item\label{rmk:20230830_ii} The above bound allows to optimize the initial stepsize $\gamma$, which reaches the minimum at
\begin{equation*}
\gamma = D \bigg( \frac 4 3 \bar{\beta}^2 + \sqrt{2}a_p\bar{\beta}\sqrt{\frac{\sigma^2}{m}+L^2}  + \frac{\sigma^2/m + L^2}{((p+ 1/2)\wedge 1)}  \bigg)^{-1/2}
\end{equation*}
making the final bound equal to
\begin{align*}
f(\bar{x}_k)- \min_X f &\leq D\frac{(p+1)\vee 1}{\sqrt{k}} 
\bigg[  \bigg( \frac 4 3 \bar{\beta}^2 + \sqrt{2}a_p\bar{\beta}\sqrt{\frac{\sigma^2}{m}+L^2}  + \frac{\sigma^2/m + L^2}{((p+ 1/2)\wedge 1)}  \bigg) \\[2ex]
& + \bigg(\frac 4 3 \bar{\beta}  +  \left(1+ \frac 1 \varepsilon\right)  \bigg( 
\frac{1}{((2p+1)\wedge 1)}\frac{4\sigma}{\sqrt{2 m}} +
 \frac{1}{(p+ 1/2)\wedge 1}\frac{\sigma^2}{\bar{\beta} m}  \bigg) \bigg)\sqrt{\log \frac 2 \delta} 
\bigg].
\end{align*}
\end{enumerate}
Again, also here it is clear the advantage of our results compared to  
\cite{Gorbunov2020,Gorbunov2021}, since the dependence on the confidence level is $\sqrt{\log(2/\delta)}$.
\end{remark}

\begin{remark}[On the expected optimization error]\label{rmk:in-expctation}\ 
Clipping is unnecessary under our noise assumptions when in-expectation rates are in order, since the average iterate of SsGM already enjoys the optimal convergence rate 
\begin{equation}
\label{ex:ssgm}
\mathbb{E}[f(\bar{x}_k) - f^*] \leq \sqrt{\frac{D^2(\sigma^2/m+L^2)}{k}} \;,
\end{equation}
which is achieved with step-size $\gamma_i = \gamma/\sqrt{i}$ and $\gamma = \sqrt{D^2/(\sigma^2/m+L^2)}$. On the other hand, in \Cref{appC} we could easily derive from our analysis the following in-expectation bound for the clipped-SsGM
\begin{equation}
\label{ex:cssgm}
\Ex[f(\bar{x}_k) - f^*] \leq 
\sqrt{\frac{D^2(\sigma^2/m + L^2)}{k}} + \frac{\sigma^2}{m} \left(1+\frac{1}{\varepsilon}\right) \sqrt{\frac{D^2}{\beta^2 k}} \;.
\end{equation}
We make the following comments: (1) SsGM and C-SsGM share the same optimal step-size; (2) the rate in equation \eqref{ex:cssgm} is worse than that is \eqref{ex:ssgm}, since it exhibits an additional term due to the bias in the subgradient estimators, a fact that we noted also in the experiments under light-tails noise; (3) finally, when $\sigma = 0$, clipping never occurs regardless of $\beta$, and C-SsGM behaves exactly as the deterministic sugradient method and the above bound reduces to
\begin{equation}
f(\bar{x}_k) - f^* \leq \frac{DL}{\sqrt{k}} \;,
\end{equation}
which is the same rate featured by the deterministic subgradient method.
\end{remark}

\begin{remark}[Behaviour under sub-Gaussian noise]\label{rmk:SGnoise}\ 
When the noise is sub-Gaussian, clipping is not necessary, and the classical SsGM method already enjoys the optimal high probability convergence rate. For example, an adaptation of the proof (reported in \Cref{appC}) of Proposition 4.1. in \cite{Lan2020}, along with the choice of $\gamma_i = \gamma/\sqrt{i}$ where
\[
\gamma = \sqrt{\frac{D^2}{4(\eta^2(1+\ln(2/\delta))+L^2)}}
\]
leads to the following high-probability convergence rate
\begin{equation}
\label{eq:SGrate}
f(\bar{x}_k)-f^* \leq \frac{2D}{\sqrt{k}} \left(L +  \eta \sqrt{\log\left(\frac{2e}{\delta}\right)} \right),
\end{equation}
where $\eta > 0$ is the variance proxy parameter of the sub-Gaussian noise. Concerning the comparison with our bounds in Corollary~\ref{cor:2023829}, the following considerations are in order:
\begin{enumerate}[leftmargin=.75cm, label={\rm (\roman*)}]
\item\label{rmk:SGN_i} Our bound in Corollary~\ref{cor:2023829} (which does not take advantage of the sub-Gaussian assumption), after the optimization of the stepsize $\gamma$, shows a worse dependence on the algorithm parameters with respect to \eqref{eq:SGrate}. Furthermore, \cref{eq:SGrate} features a sub-Gaussian behaviour for all $k$, while in the case of clipping this is only obtained for $k \gtrsim \ln(\delta^{-1})$. 
\item\label{rmk:SGN_ii} Our bound depends directly on the variance $\sigma^2$, as opposed to \eqref{eq:SGrate} which depends on the variance proxy parameter $\eta$, which in general is larger than $\sigma$.
\item\label{rmk:SGN_iii} Clipped SsGM requires a larger number of parameters to be set with respect to SsGM. Our results, provide theoretical recipes for tuning $\gamma_i$ and $\lambda_i$, but, in practice, to obtain the optimal performances one has to resort to trial-and-error procedures to optimize such parameters.
\item\label{rmk:SGN_iv} In spite of the worse theoretical bounds, our experiments show that clipping SsGM may perform well even under sub-Gaussian noise, especially when the noise level is large compared to the true subgradients. See Section~\ref{sec:ex1}.
\end{enumerate}
\end{remark}

\section{Convergence Analysis}
\label{sec:convanal}

In this section we provide the fundamental steps of the proof of the main results of the paper. Additional details are given in Appendix~\ref{appA} and \ref{appB}.

\subsection{Part 1}

In this section we address the proof of Theorem~\ref{thm:main}.
We start by gathering the main statistical properties of the clipped subgradient estimator.

\begin{lemma}
\label{lemma:CSGproperties}
Let $x \in X$, $\lambda> L$ and define
\begin{equation*}
\bar{u} = \frac 1 m \sum_{j=1}^m \hat{u}(x,\xi_j)
\quad \text{and} \quad
\tilde{u} = \min\Big\{\frac{\lambda}{\norm{\bar{u}}}, 1\Big\} \bar{u} =
\begin{cases}
\bar{u} &\text{if } \norm{\bar{u}}\leq \lambda\\
\frac{\lambda}{\norm{\bar{u}}}\bar{u}  &\text{if } \norm{\bar{u}}>\lambda.
\end{cases}
\end{equation*}
Set $u = \Ex[\hat{u}(x,\xi_j)] = \Ex[\bar{u}]$ and suppose that $\Ex[\norm{\hat{u}(x,\xi_j) - u}^2] \leq \sigma^2$. Then, the following hold.
\begin{enumerate}[leftmargin=.75cm, label={\rm (\roman*)}]
\item\label{lemma3_i} 
$\Ex\norm{\tilde{u}}^2 \leq \Ex\norm{\bar{u}}^2\leq\dfrac{\sigma^2}{m} + L^2$ 
\quad {\rm (Second Moment)}.
\item\label{lemma3_ii} 
$\norm{\Ex\tilde{u} -  u} \leq  \dfrac{\sigma^2}{m(\lambda-L)}$\quad {\rm (Bias)}.
\item\label{lemma3_iv} 
$\Ex\norm{\tilde{u} - u}^2 \leq  
\dfrac{\sigma^2}{m}\bigg[1 + \bigg(\dfrac{\lambda+L}{\lambda-L}\bigg)^2 \bigg]$
\quad {\rm (MSE)}.
\item\label{lemma3_iii} 
$\Ex\norm{\tilde{u} - \Ex \tilde{u} }^2 \leq 
\dfrac{\sigma^2}{m}\bigg[1 + \bigg(\dfrac{\lambda+L}{\lambda-L}\bigg)^2 \bigg]$
\quad {\rm (Variance)}.
\end{enumerate}
\end{lemma}
\begin{remark}
\label{rmk:eps}
The previous bounds assume that $\lambda>L$. We can restate bounds 
\ref{lemma3_ii} and \ref{lemma3_iii} in different way. Let $\varepsilon>0$ and suppose that $\lambda \geq L_\varepsilon$. Then we have
\begin{equation*}
\norm{\Ex\tilde{u} -  u} \leq \dfrac{\sigma^2}{m}\bigg(1+\frac1 \varepsilon\bigg) \frac{1}{\lambda}
\quad\text{and}\quad
\Ex\norm{\tilde{u} - \Ex \tilde{u} }^2 \leq \frac{\sigma^2}{m} \bigg[\ 1 + \bigg(\frac{2+\varepsilon}{\varepsilon} \bigg)^2\bigg] \leq 4 \frac{\sigma^2}{m} \bigg(1+\frac{1}{\varepsilon}\bigg)^2.
\end{equation*}
\end{remark}

Now, we are ready to tackle the proof of Theorem~\ref{thm:main}. We start by decomposing the error in several possibly simpler terms.

\begin{lemma}[Decomposition of the error]
\label{lemma:ErrorDecomp}
Let $(x_k)_{k \in \N}$ and $(\bar{x}_k)_{k \in \N}$ be generated by Algorithm~\ref{algo:projsub2}. 
Suppose that the sequence $(w_k/\gamma_k)_{k \in \N}$ is increasing. Then,
for all $k \in \N$ and $x \in X$, we have
\begin{equation}
f(\bar{x}_k) - f(x) \leq 
\frac{1}{\sum_{\iter=1}^k w_i}\bigg[
\frac{D^2}{2}
\frac{w_k}{\gamma_k}
+ \underbrace{\sum_{\iter=1}^{k} \theta_i^v}_{\textbf{A}} + \underbrace{ \sum_{\iter=1}^{k} \theta_i^b}_{\textbf{B}} + \frac{1}{2}\underbrace{ \sum_{\iter=1}^{k} \zeta_i}_{\textbf{C}} + \frac{1}{2} \underbrace{\sum_{\iter=1}^{k} \nu_i}_{\textbf{D}}
\bigg] \quad\text{a.s.},
\end{equation}
where 
\begin{itemize}[leftmargin=.75cm]
\item $u_i = \Ex[\bar{u}_i\,\vert\, x_i] =\Ex[\hat{u}(x_i, \xi^i_j)\,\vert\, x_i]$\ \ {\rm(}for $1 \leq j \leq m${\rm)},
\item $\theta_i^v \coloneqq  
w_i \langle \tilde{u}_\iter - \Ex[\tilde{u}_i \,\vert\, x_1, \dots, x_i], x - x_i \rangle$,
\item $\theta_i^b \coloneqq 
w_i \langle \Ex[\tilde{u}_i \,\vert\, x_1, \dots, x_i] - u_i, x - x_\iter \rangle$, 
\item $\zeta_i \coloneqq w_i \gamma_i\big(\norm{\tilde{u}_i}^2 
-\Ex\big[\norm{\tilde{u}_i }^2\,\vert\, x_1, \dots, x_i\big]\big)$,
\item $\nu_i \coloneqq 
w_i \gamma_i\Ex\big[\norm{\tilde{u}_i }^2\,\vert\, x_1, \dots, x_i\big]$.
\end{itemize}
\end{lemma}

From Lemma~\ref{lemma:ErrorDecomp} it is clear that we have to bound with high probability the four summations we named {\bf A}, {\bf B}, {\bf C} and {\bf D}, by applying a Bernstein's type concentration inequality. So, the proof of Theorem~\ref{thm:main} goes through the following steps that we collect in separate propositions.

\begin{remark}[On the choice of the concentration inequality]
Two popular tools to control bounded sums of martingale difference sequences, such as terms {\bf A} and {\bf C}, are Azuma-Hoeffding's and Freedman's inequalities. The former can be used to derive tight high-probability bounds for SsGM under bounded noise (and even sub-Gaussian noise). Indeed, suppose $d_i$ is a martingale difference sequence with $|d_i| \leq b_i$ almost surely. Then, Azuma-Hoeffding's inequality states that, with probability at least $1-\delta$
\begin{equation}
\label{eq:AH}
\sum_{i=1}^k d_i \lesssim \sqrt{\sum_{i=1}^k b_i^2 \cdot \log\left(\frac{2}{\delta}\right)}\,.
\end{equation}
This method works well when $b_i \leq \textsc{constant}$ as in the case of SsGM, where either $d_i = \langle \hat{u}_i - u, x-x_i \rangle$ or $d_i = \|\hat{u}_i\|^2 - \mathbb{E}[\|\hat{u}_i\|^2 |x_1,\dots,x_i]$. On the other hand, in the case of the clipped SsGM, where either $d_i = \theta_i^v$, or $d_i = \zeta_i$, we have $b_i \sim w_i \sqrt{i}$. As a result, the RHS of Azuma-Hoeffding's inequality cannot be compensated by the sum of the weights $W_k = \sum_{i=1}^k w_i$. By contrast, Freedman's inequality provides the following bound
\begin{equation}
\label{eq:F}
\sum_{i=1}^k d_i \lesssim \underset{1 \leq i \leq k} \max b_i \cdot \log\left(\frac{2}{\delta}\right) + \sqrt{F \cdot \log\left(\frac{2}{\delta}\right)} \;,
\end{equation}
where $F$ is the total conditional variance of the $d_i$'s. Note that, differently from \eqref{eq:AH}, the RHS of \eqref{eq:F} only depends on the largest $b_i$ and 
this ultimately can be compensated by $W_k$ as long as $F = \mathcal{O}(W_k)$. 
\end{remark}

\begin{restatable}[Freedman's Bound]{proposition}{propFreedmanBound}
\label{prop:FreedmanBound}
Under the same assumptions of Fact~\ref{thm:freedman}, let $\delta \in \left]0, 2/e \right]$, $k \in \N$ and let $F\geq 0$ be s.t. $V_k \leq F$ a.s. Then, with probability at least than $1-\delta/2$ we have
\begin{equation*}
\sum_{\iter=1}^\Iter X_\iter \leq \frac{2}{3} c \log\left(\frac{2}{\delta}\right) + \sqrt{ 2 F\log\left(\frac{2}{\delta}\right)} \;.
\end{equation*}
\end{restatable}

\begin{restatable}[Analysis of the term {\bf A}]{proposition}{propAnalysisA}
\label{prop:AnalysisA}
Let  $k \in \N$ and $\varepsilon>0$ and suppose that, for every $i =1, \dots, k, \lambda_i \geq (1+\varepsilon) L$. Then, with probability at least $1- \delta/2$, we have 
\begin{equation}
\sum_{\iter=1}^k \theta_i^v \leq \frac{4}{3} D \cdot \max_{1 \leq \iter \leq \Iter} w_\iter \lambda_\iter \cdot \log\left(\frac{2}{\delta}\right) + 2 \bigg(1 + \frac 1 \varepsilon \bigg) \frac{D \sigma}{\sqrt{m}} \sqrt{ 2\sum_{\iter=1}^\Iter 
w_\iter^2} \cdot \sqrt{\log\left(\frac{2}{\delta}\right)} \;.
\end{equation}
\end{restatable}

\begin{restatable}[Analysis of the term {\bf B}]{proposition}{propAnalysisB}
\label{prop:AnalysisB}
Let $k \in \N$. The following hold
\begin{equation}
\sum_{\iter=1}^k \theta_i^b 
\leq \frac{D \sigma^2}{m} \bigg(1+ \frac 1 \varepsilon \bigg) \sum_{\iter=1}^\Iter \frac{w_\iter}{\lambda_\iter} \qquad\text{a.s.}
\end{equation}
\end{restatable}

\begin{restatable}[Analysis of the term {\bf C}]{proposition}{propAnalysisC}
\label{prop:AnalysisC}
Let $k \in \N$. Then, with probability at least $1-\delta/2$ 
\begin{equation*}
\sum_{\iter=1}^\Iter \zeta_i 
\leq \frac{4}{3} \cdot \max_{1 \leq \iter \leq \Iter} w_\iter \gamma_\iter \lambda^2_\iter \cdot \log\left(\frac{2}{\delta}\right) + \sqrt{2\bigg(\frac{\sigma^2}{m} + L^2 \bigg) \sum_{\iter=1}^\Iter w_\iter^2 \gamma_\iter^2 \lambda_\iter^2}\cdot \sqrt{\log\left(\frac{2}{\delta}\right)} \;.
\end{equation*}
\end{restatable}
The proof of the previous propositions are given in Appendix~\ref{appA}. The next one is a direct consequence of the definition of $\nu_i$ and the bound in Lemma~\ref{lemma:CSGproperties}\ref{lemma3_i}.
\begin{restatable}[Analysis of the term {\bf D}]{proposition}{propAnalysisD}
\label{prop:AnalysisD}
Let $k \in \N$. Then we have
\begin{equation*}
\sum_{\iter=1}^\Iter \nu_\iter \leq \bigg(\frac{\sigma^2}{m} + L^2\bigg)
\sum_{\iter=1}^\Iter w_\iter \gamma_\iter\qquad\text{a.s.}
\end{equation*}
\end{restatable}

{\bf Proof of Theorem~\ref{thm:main}.}
By a simple union bound, it follows from Proposition~\ref{prop:AnalysisA} and Proposition~\ref{prop:AnalysisC} that with probability at least $1 - \delta$ we have 
\begin{align*}
\sum_{\iter=1}^k \theta_i^v + \frac 1 2 \sum_{\iter=1}^\Iter \zeta_i  
&\leq
\frac{2}{3} \bigg(2 D \cdot \max_{1 \leq \iter \leq \Iter} w_\iter \lambda_\iter + \max_{1 \leq \iter \leq \Iter} w_\iter \gamma_\iter \lambda^2_\iter \bigg) \log \bigg(\frac 2 \delta\bigg)\\
&+ \frac{1}{\sqrt{2}}\bigg(4 \bigg(1 + \frac 1 \varepsilon \bigg) \frac{D \sigma}{\sqrt{m}} \sqrt{\sum_{\iter=1}^\Iter 
w_\iter^2} + \sqrt{\bigg(\frac{\sigma^2}{m} + L^2 \bigg) \sum_{\iter=1}^\Iter w_\iter^2 \gamma_\iter^2 \lambda_\iter^2} \bigg)\sqrt{\log\left(\frac{2}{\delta}\right)}.
\end{align*}
Since the bounds on the terms {\bf B} and {\bf D} hold almost surely, the statement follows directly by plugging the above bound,
and those given in Proposition~\ref{prop:AnalysisB} and Proposition~\ref{prop:AnalysisD} into the inequality of  Lemma~\ref{lemma:ErrorDecomp}.
\qed

\subsection{Part 2}
In this section we provide the proof of Theorem~\ref{thm:main2}.
We start with the following lemmas whose proofs are provided in the Appendix~\ref{appB}.
\begin{lemma}
\label{lem:harmonicsums}
Let $p \in \R$. Then
\begin{equation*}
\begin{cases}
\displaystyle
\frac{k^{p+1}}{(p+1)\vee 1} \leq \sum_{i=1}^k i^p \leq 
\frac{k^{p+1}}{(p+1)\wedge 1} &\text{if } p >-1\\[1ex]
\displaystyle
\frac 1 2 + \log k \leq \sum_{i=1}^k i^p \leq 1 + \log k &\text{if } p =-1\\[1ex]
\displaystyle
1 \leq \sum_{i=1}^k i^p \leq \frac{p}{p+1}&\text{if } p <-1 
\end{cases}
\end{equation*}
\end{lemma}
\begin{lemma}
\label{lem:maxseq}
Let $b,c>0$ and $t,s \in \R$. Then, for every integer $k \geq 1$, we have
\begin{equation*}
\max_{1 \leq i \leq k} \max\{b\, i^t, c\, i^s\}
= \max\{b\, k^{t_+}, c\, k^{s_+}\}.
\end{equation*}
\end{lemma}
\begin{lemma}
\label{lem:summax}
Let $b,c>0$ and $t,s \in \R$ with $t>s$. Then, for every integer $k \geq 1$, we have
\begin{equation*}
\sum_{i=1}^k \max\{b\, i^t, c\, i^s\}
\leq
\begin{cases}
\displaystyle
\bigg(\frac{1}{(s+1)\wedge 1} + \frac{1}{(t+1)\wedge 1}\bigg) \max\bigg\{b, \frac{c}{k^{t-s}} \bigg\} \cdot k^{t+1}&\text{if } -1<s<t\\[3ex]
\displaystyle
\frac{k^{t+1}}{(t+1)\wedge 1} \cdot
\begin{cases}
c  &\text{if } k< \left(\dfrac c b\right)^{\frac{1}{t-s}} \\[2ex]
b  \bigg( \dfrac{(c/b)^{2}}{k^{t-s}} +1 \bigg) &\text{if }  k \geq \left(\dfrac c b\right)^{\frac{1}{t-s}}
\end{cases} &\text{if } -1=s<t\\[7ex]
\displaystyle
\frac{c\, s}{s+1} + \max\bigg\{b, \frac{c}{k^{t-s}} \bigg\} \frac{k^{t+1}}{(t+1)\wedge 1} &\text{if } s<-1<t\\[3ex]
\displaystyle
\frac{c\,s}{s+1} + b(1 + \log k) &\text{if } s<t=-1\\[3ex]
\displaystyle
\frac{c\,s}{s+1} + \frac{b\,t}{t+1} &\text{if } s<t<1\\[3ex]
\end{cases}
\end{equation*}
\end{lemma}
With the help of the above results it is easy to conduct
the analysis of the terms in \eqref{20230822a} and \eqref{20230822b}. First of all, according to Theorem~\ref{thm:main}, in order to ensure that
$(w_k/\gamma_k)_{k \in \N}$ is increasing, we have $p+r\geq 0$.
Concerning the first quantity in \eqref{20230822b}, it follows from Lemma~\ref{lem:harmonicsums} that
\begin{equation}
\label{eq:term1}
\circled{4} =\frac{\sqrt{\sum_{i=1}^k w_i^2}}{\sum_{i=1}^k w_i} \leq
\begin{cases}
\displaystyle
\frac{(p+1)\vee 1}{(2p+1)\wedge 1}\cdot\frac{1}{\sqrt{k}}
&\text{if } -1/2 <p\\[3ex]
\displaystyle
((p+1)\vee 1)\cdot\frac{\sqrt{1+\log k}}{\sqrt{k}} &\text{if } p=-1/2\\[2ex]
\displaystyle
\sqrt{\frac{2p}{2p+1}} \cdot \frac{1}{k^{p+1}}&\text{if } -1<p<-1/2\\[2ex]
\displaystyle
\frac{\sqrt{2}}{\log k + 1/2}&\text{if } p=-1\\[2ex]
\displaystyle
\frac{2p}{2p+1}&\text{if } p<-1.
\end{cases}
\end{equation}
This implies that the left hand side will converge to zero provided that $p\geq -1$. In particular, if $p>-1/2$ we have the best possible rate of $O(1/\sqrt{k})$. In the following we will assume $p>-1$ (since for $p=1$ we have a slow rate, not even polynomial).

Next we address the last two terms in \eqref{20230822b}. We have
\begin{equation}
\label{eq:term2}
    \circled{7} =\frac{\sum_{i=1}^k w_i \gamma_i}{\sum_{i=1}^k w_i}
    \leq \gamma ((p+1)\vee 1)\cdot
    \begin{cases}
    \displaystyle
    \frac{1}{(p-r+1)\wedge 1}\cdot \frac{1}{k^r}
    &\text{if } p-r>-1\\[3ex]
    \displaystyle
     \frac{1+\log k}{k^{r}} &\text{if } p-r=-1\\[2ex]
    \displaystyle
    \frac{p-r}{p-r+1}\cdot \frac{1}{k^{p+1}}
    &\text{if } p-r<-1
    \end{cases}
\end{equation}
and
\begin{equation}
\label{eq:term3}
    \circled{6} =\frac{\sum_{i=1}^k w_i/\lambda_i}{\sum_{i=1}^k w_i}
    \leq \frac{(p+1)\vee 1}{\beta}\cdot
    \begin{cases}
    \displaystyle
    \frac{1}{(p-q+1) \wedge 1}\cdot \frac{1}{k^q}
    &\text{if } p-q>-1\\[3ex]
    \displaystyle
     \frac{1+\log k}{k^{q}} &\text{if } p-q=-1\\[2ex]
    \displaystyle
    \frac{p-q}{p-q+1}\cdot \frac{1}{k^{q+1}}
    &\text{if } p-q<-1.
    \end{cases}
\end{equation}
Since we assumed $p>-1$, from the previous bounds, it is easy to see that, in order to have both terms converging to zero we should require in every cases that
$r>0$ and $q>0$.
Now we consider the three terms in \eqref{20230822a}. Concerning the first one we have
\begin{equation}
\label{eq:term4}
    \circled{1} =\frac{w_k/\gamma_k}{\sum_{i=1}^k w_i} \leq \frac{(p+1)\vee 1}{\gamma}\cdot \frac{1}{k^{1-r}}.
\end{equation}
This shows that necessarily $r<1$.
Instead for the other two quantities, recalling Lemma~\ref{lem:maxseq}, we have
\begin{equation}
\label{eq:term5}
    \circled{2} =\frac{\max_{1\leq i \leq k} w_i \lambda_i}{\sum_{i=1}^k w_i} \leq \frac{(p+1)\vee 1}{k^{\min\{p+1,1-q\}}} \cdot
    \max\bigg\{ \beta, \frac{(1+\varepsilon) L}{k^{\min\{(p+q)_+,q\}}}\bigg\}.
\end{equation}
and
\begin{equation}
\label{eq:term6}
    \circled{3} =\frac{\max_{1\leq i \leq k} w_i \gamma_i \lambda_i^2}{\sum_{i=1}^k w_i} \leq \frac{\gamma((p+1)\vee 1)}{k^{\min\{p+1,r+1-2q\}}} \cdot
    \max\bigg\{ \beta^2, \frac{L^2_\varepsilon}{k^{\min\{(2q+p-r)_+,2q\}}}\bigg\}.
\end{equation}
From the above bounds, since $p+1 > 0$, it is clear that
in order to have convergence to zero we should impose that
\begin{equation*}
    1-q>0\quad \text{and}\quad r+1-2q>0.
\end{equation*}
Overall, up to now, the convergence of the considered terms is ensured if
\begin{equation*}
p>-r,\ r \in \left]0,1\right[,\ q \in \left]0,1\right[,\ \text{and}\ q<\frac{r+1}{2}.
\end{equation*}
Moreover, from \eqref{eq:term2} and \eqref{eq:term4} it follows that the optimal choice of $r$ in term of rate of convergence is $r=1/2$. Finally we consider the second term in \eqref{20230822b}. Recalling Lemma~\ref{lem:summax}, and setting for the sake of brevity 
\begin{equation*}
L_\varepsilon = L_\varepsilon,\quad
A_{\varepsilon, \beta} = \frac{L_\varepsilon}{\beta}\quad\text{and}\quad B_{p,r} = \frac{(2p-2r)}{2p-2r+1},
\end{equation*}
 we have
\begin{equation}
\sum_{i=1}^k w_i^2 \gamma_i^2 \lambda_i^2
\leq \gamma^2\cdot
\begin{cases}
\displaystyle
\frac{2}{(2p-2r+1)\wedge 1} 
\max\bigg\{\beta, \frac{L_\varepsilon}{k^{q}} \bigg\}^2 \cdot k^{2p-2r+2q+1}&\text{if } -\frac 1 2<p-r\\[3ex]
\displaystyle
\frac{k^{2q}}{(2p\!-\!2r\!+\!2q\!+1)\wedge 1} \cdot
\begin{cases}
L^2_\varepsilon  &\text{if } k< A_{\varepsilon,\beta}^{1/q} \\[2ex]
\beta^2  \bigg(  \dfrac{A_{\varepsilon,\beta}^{4}}{k^{2q}} \!\!\! +1 \bigg) &\text{if }  k \geq A_{\varepsilon,\beta}^{1/q}
\end{cases} &\text{if } - \frac 1 2=p-r\\[7ex]
\displaystyle
L^2_\varepsilon B_{p,r} + \max\bigg\{\beta, \frac{L_\varepsilon}{k^{q}} \bigg\}^2 \frac{k^{2p-2r+2q+1}}{(2p-2r+2q+1)\wedge 1} &\text{if }\ \substack{\displaystyle p-r<-\frac 1 2\\[1ex]\displaystyle<p-r+q}\\[5ex]
\displaystyle
L^2_\varepsilon B_{p,r} + \beta^2(1 + \log k) &\text{if } p-r +q=-\frac 1 2\\[3ex]
\displaystyle
L^2_\varepsilon B_{p,r} + \frac{\beta^2(2p-2r+2q)}{2p-2r+2q+1} &\text{if } p-r+q<- \frac 1 2\\[3ex]
\end{cases}
\end{equation}
and hence
\begin{multline*}
\circled{5} =\frac{\sqrt{\sum_{i=1}^k w_i^2 \gamma_i^2 \lambda_i^2}}{\sum_{i=1}^k w_i} \leq
\gamma\cdot ((p+1)\vee 1)\\[2ex]
\times
\begin{cases}
\displaystyle
\sqrt{\frac{2}{(2p-2r+1)\wedge 1}} 
\max\bigg\{\beta, \frac{L_\varepsilon}{k^{q}} \bigg\} \cdot \frac{1}{k^{r+\frac 1 2 -q}} &\text{if } -\frac 1 2<p-r\\[3ex]
\displaystyle
\frac{1}{\sqrt{(2q)\wedge 1}} \cdot \frac{1}{k^{r+\frac 1 2 -q}} \cdot
\begin{cases}
(1+\varepsilon) L  &\text{if } k< A_{\varepsilon,\beta}^{1/q} \\[2ex]
\beta  \bigg[ \dfrac{A^2_{\varepsilon,\beta}}{k^q} \!\!\! +1 \bigg] &\text{if }  k \geq A_{\varepsilon,\beta}^{1/q}
\end{cases} &\text{if } - \frac 1 2=p-r\\[7ex]
\displaystyle
\bigg(\frac{L_\varepsilon\sqrt{B_{p,r}}}{k^{p+q-r+\frac 1 2}} + \max\bigg\{\beta, \frac{L_\varepsilon}{k^{q}} \bigg\} \frac{1}{\sqrt{(2p-2r+2q+1)\wedge 1}}\bigg) \frac{1}{k^{r+\frac 1 2 -q}} &\text{if }\ \substack{\displaystyle p-r<-\frac 1 2\\[1ex]\displaystyle<p-r+q}\\[5ex]
\displaystyle
\big(L_\varepsilon\sqrt{B_{p,r}} + \beta \sqrt{1 + \log k}\big)\cdot\frac{1}{k^{r+\frac 1 2 -q}} &\text{if } p-r +q=-\frac 1 2\\[3ex]
\displaystyle
\bigg(L_\varepsilon\sqrt{B_{p,r}} + \beta \sqrt{\frac{2p-2r+2q}{2p-2r+2q+1}}\bigg)\cdot \frac{1}{k^{p+1}} &\text{if } p-r+q<- \frac 1 2\\[3ex]
\end{cases}
\end{multline*}
This bound shows that convergence is guaranteed if $r+1/2 - q>0$, that is $q<r+1/2$. Taking into account that $r>0$, this provide a weaker condition than the condition $q<(r+1)/2$ already obtained. 
Moreover, with the optimal choice of $r=1/2$, \circled{5}
features a rate of $1/k^{1-q}$, which, considering that $q>0$, is optimaized when $q=1/2$. In the end Theorem~\ref{thm:main2}\ref{thm:main2_i}\ref{thm:main2_ii} follows by simply plugging the above bounds with $q=r=1/2$ in Theorem~\ref{thm:main}. Finally concerning Theorem~\ref{thm:main2}\ref{thm:main2_iii}, we note that if we take the stepsizes constant till $k$, that is $(\gamma_i)_{1 \leq i \leq k} \equiv \gamma$ (so that $r=0$), then we should ask for $p\geq 0$ and the bounds for \circled{1}, \circled{3}, \circled{5}, and \circled{7} (with $q=1/2$) become
\begin{align*}
\circled{1} &
\leq \frac{p+1}{\gamma}\cdot \frac{1}{k},\quad \circled{3} \leq \gamma (p+1)\cdot \max\bigg\{\beta^2, \frac{L^2_\varepsilon}{k}\bigg\}\\[1ex] 
\circled{5} & \leq \gamma (p+1)\sqrt{2} 
\max\bigg\{\beta, \frac{L_\varepsilon}{\sqrt{k}}\bigg\} \frac{1}{\sqrt{k}},\quad
\circled{7} \leq \gamma (p+1),
\end{align*}
Thus, if we choose $(\gamma_i)_{1 \leq i \leq k} \equiv \gamma/\sqrt{k}$ and substitute the bounds in Theorem~\ref{thm:main}, the statement of Theorem~\ref{thm:main2}\ref{thm:main2_iii} follows.

\section{Application to Kernel Methods}
\label{sec:kenels}
In this section we present an application of \Cref{algo:projsub2} to the case of kernel methods
and show that the clipping strategy can be implemented in this setting by updating finite dimensional variables.
This resembles the classical kernel trick which is common in empirical risk minimization \cite{Steinwart2008} and in SGD as well \cite{Dieuleveut2016,Lin2016a,Lin2016b}, Here we adjust the method so to handle the nonlinearity of the clipping operation. In passing we note that our weak assumption on the noise allows to treat unbounded kernels as the polynomial kernel and contrasts with the common bounded kernel assumption \cite{Caponnetto2007,Steinwart2009}. 

The problem we address is formulated as follows
\begin{equation*}
\label{eq:statlearnprobl}
\min_{x \in B_r} R(x) \coloneqq \Ex[\ell(\langle x, \Phi(Z) \rangle, Y)],
\end{equation*}
where $Z$ and $Y$ are random variables with values in the measurable spaces $\cZ$ and $\Y$ respectively, with joint distribution $\mu$;
$H$ is a Hilbert space with scalar product $\langle \cdot, \cdot \rangle$ and $B_r$ is the ball of radius $r > 0$ centered at the origin of $H$. We assume the \emph{loss function} $\ell: \mathbb{R} \times \Y \rightarrow \R_+$ to be convex and $L$-Lipschitz in the first argument (for every fixed value of the second argument) and we assume that $\Ex[\ell(0, Y)]<+\infty$. The function 
$\Phi: \cZ \rightarrow H$ is the {\it feature map} and $K(\cdot,\cdot) = \langle \Phi(\cdot), \Phi(\cdot) \rangle$ is the corresponding kernel. We further assume that $\Ex[\norm{\Phi(Z)}^2]<+\infty$. The goal is to find a function $\cZ \ni z \mapsto \langle x, \Phi(z)\rangle \in \R$, with $x \in H$, 
 which minimizes the expected risk above over the ball $B_r$, based on the possibility of sampling from the distribution $\mu$.

\paragraph{Derivation} Assume $x_0=0$ and recall that, for each $k \geq 0$, the main update in \Cref{algo:projsub2} is defined by the following
\begin{equation}
x_{k+1} = P_{B_r}(x_k - \gamma_k \tilde{u}_k),
\label{eq:clipped-SsGM}
\end{equation}
where $\tilde{u}_k = \textsc{CLIP}(\bar{u}_k,\lambda_k)$, $\bar{u}_k = \frac{1}{m} \sum_{j=1}^m \hat{u}(x_k, (Z_j^k, Y_j^k))$, $\hat{u}(x, (Z, Y)) = \ell'(\langle x, \Phi(Z) \rangle, Y) \Phi(Z)$, where, for all $(t,y) \in \mathbb{R}\times \Y$, $\ell'(t,y)$ is a subgradient of $\ell(\cdot, y)$ at $t$, that is, $\ell'(t,y) \in \partial\ell(t, y)$. Since $H$ may be infinite dimensional, computing $\tilde{u}_k$ with a straightforward application of its definition may be problematic.

In order to solve the aforementioned issue, the algorithm will keep an implicit representation of the iterates $x_k$'s in terms of the kernels. To obtain this representation, we start by noticing that since $P_{B_r}$ is the projection in the ball centered at the origin, then
\begin{equation}
\label{eq:kernel_up}
x_{k+1} = \min \left\{\frac{r}{\|x_k-\gamma_k \tilde{u}_k\|}, 1 \right\} (x_k - \gamma_k \tilde{u}_k) = \delta_k (x_k - \gamma_k \tilde{u}_k),
\end{equation}
where we set $\delta_k = \min \left\{\frac{r}{\|x_k-\gamma_k \tilde{u}_k\|}, 1 \right\}$. Similarly by using the definitions of clipping and $\hat{u}(x, (Z, Y))$ it holds that
\begin{equation}
\label{eq:kernel_clip}
\tilde{u}_{k} = \min \left\{\frac{\lambda_k}{\|\bar{u}_k\|}, 1 \right\} \bar{u}_k = \frac{\rho_k}{m} \sum_{j=1}^m \hat{u}(x_k, (Z_j^k, Y_j^k)) = \frac{\rho_k}{m} \sum_{j=1}^m \alpha_j^k \Phi(Z_j^k),
\end{equation}
where we set $\rho_k = \min \left\{\frac{\lambda_k}{\|\bar{u}_k\|}, 1 \right\}$ and $\alpha_j^k = \ell'(\langle x_k, Z_j^k \rangle , Y_j^k)$. From equation \eqref{eq:kernel_clip} we have
\begin{equation}
\|\bar{u}_k\|^2 = \frac{1}{m^2} \sum_{j,j'=1}^m \alpha_j^k \alpha_{j'}^k K(Z_j^k, Z_{j'}^k).
\end{equation}
The above equation allows for the computation of $\rho_k$ once the $\alpha_j^k$s are known. Furthermore, combing  \Cref{eq:kernel_clip} and \Cref{eq:kernel_up} we obtain that
\begin{align}
\label{eq:kernel_wk}
x_{k+1} &= \delta_k (x_k - \gamma_k \tilde{u}_k) = \delta_k \left(x_k - \gamma_k \frac{\rho_k}{n} \sum_{j=1}^n \alpha_j^k\Phi(Z_j^k)\right) = \delta_k x_k - \frac{\delta_k\gamma_k\rho_k}{n} \sum_{j=1}^n \alpha_j^k\Phi(Z_j^k) \nonumber\\
&= \delta_k x_k + \sum_{j=1}^n \left(- \frac{\delta_k\gamma_k\rho_k}{n} \right) \alpha_j^k\Phi(Z_j^k) = \sum_{i=0}^{k} \sum_{j=1}^n a_{ij}^k \Phi(Z_j^i) \:,
\end{align}
where we set
\begin{equation}
a_{ij}^k = \begin{cases}
\delta_k a_{ij}^{k-1} & \text{ if } i \leq k-1,\\[1ex]
-\dfrac{\delta_k\gamma_k\rho_k}{n} \alpha_j^k & \text{ if } i = k.
\end{cases}    
\end{equation}
As a consequence
\begin{align*}
\|x_{k} - \gamma_k \tilde{u}_k\|^2 &= \|x_k\|^2 + \gamma_k^2 \|\tilde{u}_k\|^2 - 2 \gamma_k \langle x_k, \tilde{u}_k \rangle \\[1ex]
&= \|x_k\|^2 + \gamma_k^2 \rho_k^2 \|\bar{u}_k\|^2 - 2 \gamma_k \rho_k \langle x_k, \bar{u}_k \rangle \\
&= \|x_k\|^2 + \gamma_k^2 \rho_k^2 \|\bar{u}_k\|^2 - 2 \frac{\gamma_k \rho_k}{n} \langle x_k, \sum_{j'=1}^n \alpha_{j'}^k \Phi(Z_{j'}^k) \rangle \\
&= \|x_k\|^2 + \gamma_k^2 \rho_k^2 \|\bar{u}_k\|^2 - 2 \frac{\gamma_k \rho_k}{n} \sum_{i=0}^{k-1}\sum_{j=0}^n \sum_{j'=1}^n a_{ij}^{k-1} \alpha_{j'}^k K(Z_j^i, Z_{j'}^k). 
\end{align*}
Note that the above equation allows for the computation of $\delta_k$ once the $\alpha_j^k$s (and then the $a_{ij}^k$s) are known. Replacing the expression of $x_k$ from \eqref{eq:kernel_wk} in the definition of $\alpha_j^k$ leads to
\begin{align}
\alpha_j^k = \ell'(\langle x_k, \Phi(Z_j^k)\rangle, Y_j^k) = \ell' \left(\sum_{i=0}^{k-1} \sum_{j'=1}^m a_{ij'}^{k-1} K(Z_{j'}^i, Z_j^k), Y_j^k\right),
\end{align}
which can be computed directly using the kernels.

\paragraph{Algorithm} As mentioned in the previous paragraph, the algorithm keeps an implicit representation for the iterates $x_k$ which is computed as follows. This is enough to make predictions on new points as we are going to show in the following. We notice that, at any time $k$, it is possible to make a prediction for an instance $Z$ using the $k$-th iterate as follows
\begin{align}
\label{eq:kernel_pred}
\langle x_{k+1}, \Phi(Z) \rangle &= \delta_{k} \langle x_{k}, \Phi(Z) \rangle - \frac{\delta_k\gamma_k\rho_k}{m} \sum_{j=1}^m \alpha_j^k K(Z_j^k, Z).
\end{align}
This prediction requires nothing but the prediction made with the previous iterate, $\delta_k, \rho_k, \alpha_j^k$s and the values of the kernels $K(Z_j^k,Z)$. It is easy to observe that, by recursion, there is no need to have an explicit expression for the $x_k$; instead it is enough to compute, along the iterations, only $\delta_k, \rho_k, a_{ij}^k$ and $\alpha_j^k$. The algorithm to compute the iterates $x_k$ in the implicit form is given in \Cref{algo:kernelSsGM}. Now, since we have developed the theory for weighted average schemes, let $(w_k)_{k \in \N}$ a sequence of non-negative weight, and let define
\begin{equation}
\label{eq:recweights}
W_{k+1} = \begin{cases}
w_1, & \text{ if } k=0, \\
W_k + w_{k+1} & \text{ otherwise}.
\end{cases}
\end{equation}
Then, by considering $\bar{x}_k$ as defined in \Cref{algo:projsub2}, the prediction  can be computed as
\begin{align}
\label{eq:kernel_w_pred}
\langle \bar{x}_{k+1}, \Phi(Z) \rangle &= \frac{1}{W_{k+1}}\left(W_k \langle \bar{x}_{k}, \Phi(Z) \rangle + \langle x_{k+1}, \Phi(Z) \rangle \right).
\end{align}

\begin{remark}
Notice that the prediction made with the $k+1$-th iterate in \Cref{eq:kernel_pred} can be computed recursively from the prediction made by the $k$-th iterate. Similarly, the prediction made by the $k+1$-th weighted average in \Cref{eq:kernel_w_pred} can be computed by the prediction made by the previous weighted average. Both there recursion allows for significant computational saving, as at each step it only necessary to compute the kernels among $Z$ and the instances of the $k$-th batch.
\end{remark}

\begin{remark}
We note that in this setting
\begin{align}
\label{eq:kernel_var}
\mathbb{E}[\|\hat{u}(x, (Z,Y)) - \mathbb{E}[\hat{u}(x, (Z,Y))]\|^2] &\leq \mathbb{E}[\|\hat{u}(x, (Z,Y))\|^2] \nonumber \\
&= \mathbb{E}[|\ell'(\langle x, Z \rangle,Y)|^2 \|\Phi(Z)\|^2] \nonumber \\
&\leq L^2 \mathcal{K}
\end{align}
where, $\mathcal{K} \coloneqq \mathbb{E}[K(Z, Z)]$ and the last line follows from the Lipschitzness of the loss. The important consequence of \Cref{eq:kernel_var} is that in order to obtain a convergence rate it is enough to assume that the expected value of the kernel is bounded.
\end{remark}

\begin{algorithm}
\caption{Kernel Clipped Stochastic subGradient Method}
Given the step-sizes $(\gamma_k)_{k \in \N} \in \R_{++}^\N$, the weights $(w_k)_{k \in \N} \in \R_{++}^\N$, the clipping levels $(\lambda_k)_{k \in \N} \in \R_{++}^\N$, the batch size $m \in \N$, $m\geq 1$, and an initial point $x_0 = 0$, do the following.
\begin{equation}
\label{eq:loop}
\begin{array}{l}
\nonumber
\textsc{Initialization}(\gamma_0,\lambda_0,m)\\[1ex]    
\text{for}\;k=1,\ldots\\[1ex]
\left\lfloor
\begin{array}{l}
\text{draw } (Z_{j}^k, Y_j^k)_{1\leq j \leq m} \ m \text{ independent copies of } (Z,Y),\\[1ex]
\text{pick } \alpha_j^k = \ell' \left( \sum_{i=0}^{k-1}\sum_{j'=1}^{m} a_{ij'}^{k-1} K(Z_{j'}^i, Z_j^k), Y_j^k \right), \text{ for each } 1 \leq j \leq m,\\[1ex]
\text{compute}\\[1ex]
\left\lfloor
\begin{array}{l}
\|\bar{u}_k\|^2 = \frac{1}{n^2} \sum_{j,j' = 1}^{m} \alpha_j^k \alpha_{j'}^k K(Z_j^k, Z_{j'}^k),\\[1ex]
\rho_k = \min\left\{\frac{\lambda_k}{\|\bar{u}_k\|}, 1\right\},\\[1ex]
\|x_k - \gamma_k \tilde{u}_k\|^2 = \|x_k\|^2 + \gamma_k^2\rho_k^2 \|\bar{u}_k\|^2 - 2 \frac{\gamma_k\rho_k}{m} \sum_{i=0}^{k-1}\sum_{j,j'=1}^{m} a_{ij}^{k-1} \alpha_{j'}^k K(Z_j^i, Z_{j'}^k),\\[1ex]
\delta_k = \min\left\{\frac{r}{\gamma_k\rho_k\|x_k - \gamma_k \tilde{u}_k\|}, 1\right\},\\[1ex]
a_{ij}^k = 
\begin{cases}
\delta_k a_{ij}^{k-1}, & \text{ if } i \leq k-1, \\
-\frac{\delta_k\gamma_k\rho_k}{m} \alpha_{j}^k & \text{ otherwise}.
\end{cases}
, 1 \leq j \leq m,\\[1ex]
\|x_{k+1}\| = \delta_k \|x_k - \gamma_k \tilde{u}_k\|,\\[1ex]
\end{array}
\right.
\end{array}
\right.
\end{array}
\end{equation}
From the sequence $(x_k)_{k \in \N}$ one defines also
\begin{equation}
(\forall\, k \in \N)\qquad\bar{x}_k = \bigg( \sum_{\iter=1}^k w_i\bigg)^{-1} \sum_{\iter=1}^k w_i x_i.
\end{equation}
\label{algo:kernelSsGM}
\end{algorithm}

\begin{algorithm}
\caption{\textsc{Initialization}}
Given the step-size $\gamma$, clipping level $\lambda$ and the batch size $m \in \N$, $m\geq 1$,
\begin{equation}
\label{eq:init1}
\begin{array}{l}
\text{draw } (Z_{j}^0, Y_j^0)_{1\leq j \leq m} \ m \text{ independent copies of } (Z,Y),\\[1ex]
\text{pick } \alpha_j^0 = \ell' (0, Y_j^0), \text{ for each } 1 \leq j \leq m,\\[1ex]
\text{compute}\\[1ex]
\left\lfloor
\begin{array}{l}
\|\bar{u}_0\|^2 = \frac{1}{n^2} \sum_{j,j' = 1}^{m} \alpha_j^0 \alpha_{j'}^0 K(Z_j^0, Z_{j'}^0),\\[1ex]
\rho_0 = \min\left\{\frac{\lambda_0}{\|\bar{u}_0\|}, 1\right\},\\[1ex]
\delta_0 = \min\left\{\frac{r}{\gamma_0\rho_0\|\bar{u}_0\|}, 1\right\},\\[1ex]
a_{0,j}^0 = -\frac{\delta_0 \gamma_0 \rho_0}{m} \alpha_j^0, 1 \leq j \leq m,\\[1ex]
\|x_1\| = \gamma_0 \delta_0\rho_0\|\bar{u}_0\|.
\end{array}
\right.
\end{array}
\end{equation}
\label{algo:init}
\end{algorithm}

\section{Numerical Experiments}
\label{sec:experiments}
In this section we present four experiments on synthetic problems, in order to show:
\begin{itemize}
\item  a comparison between Clipped-SsGM and standard SsGM under 
heavy-tailed and light-tailed noise;
\item a comparison between different averaging schemes in infinite and finite-horizon settings for Clipped-SsGM;
\item the advantages and practicability of our fully flexible parameter setting over those prescribed in \cite{Gorbunov2021};
\item a preliminary experiment on the kernel-based method to learn a non linear classifier.
\end{itemize}

\subsection{Comparison with standard SsGM under light and heavy tails noise}\label{sec:ex1} 
The goal of this experiment is to compare the performance of clipped-SsGM against the standard SsGM, under different noise regimes.

We consider the minimization of $f(x)=|x|$ over $[-1/2,1/2]$. The sub-gradient oracle returns $\textsc{sign}(x) + n$ where: $\textsc{sign}(x) = 1$ for $x \leq 0$ and $-1$ otherwise, and $n$ is either a Pareto random variable with shape parameter $2.1$, or a  Gaussian random variables, both with zero-mean and unit variance.
Note that the considered Pareto noise has no moment of order greater than $2.1$ and thus fits well our heavy-tailed noise assumption. We run SsGM and clipped-SsGM for $1000$ iterations. We set the clipping level parameters $\beta$ and $\epsilon$ to $0.01$ and $0.001$ respectively. As for the step-size, we set $\gamma_k = \gamma/\sqrt{k}$, where $\gamma$ is computed via a grid search over $\{0.01,0.05,0.1,0.2,0.3,0.5,0.6,1\}$. We note that among the previous values we have included the optimal values of $\gamma$ for the in-expectation and the high-probability bounds as discussed in \Cref{rmk:in-expctation,rmk:SGnoise}, and item \ref{rmk:20230830_ii} in \Cref{rmk:20230830}. For both algorithms we consider the standard weights $w_k=1$ and batch size $m=1$. We ran both methods 1000 times and collected the average and the $99$th percentile of the optimization error. 
The results, reported in \Cref{fig:ssgm_fail_new}, show that, under Pareto noise, the clipping strategy provides a notable acceleration of the convergence, and at the same time a significant reduction of the upper deviations from the mean (note that the $y$ axis is in $\log$ scale). 
We note that this acceleration occurs even in expectation, which is not quite aligned with the theoretical bounds given in Remark~\ref{rmk:in-expctation}.
Moreover, under Gaussian noise with large variance (\Cref{fig:ssgm_fail_new}, center), clipped-SsGM still performs well against the standard SsGM, despite the fact that the clipping is in principle not needed for the convergence in high probability (see Remark~\ref{rmk:SGnoise}) and ultimately introduces bias.
We believe that, since in the above cases the standard deviation of the noise is large compared to the true subgradients, the reduction of the variance obtained by the clipping strategy helps and may compensate the inherent distortion. 
We also tested the algorithms under a Gaussian noise with smaller variance, i.e., with $\sigma=0.1$. In this case SsGM performs better than C-SsGM as we may expect from the theory. Finally, we note that our choice of $\beta$ corresponds to a clipping level of $(1+\varepsilon)L$ across all the iterations. We checked that this aggressive schedule leads to the best performances.

\begin{figure}
\centering
\includegraphics[scale=0.3]{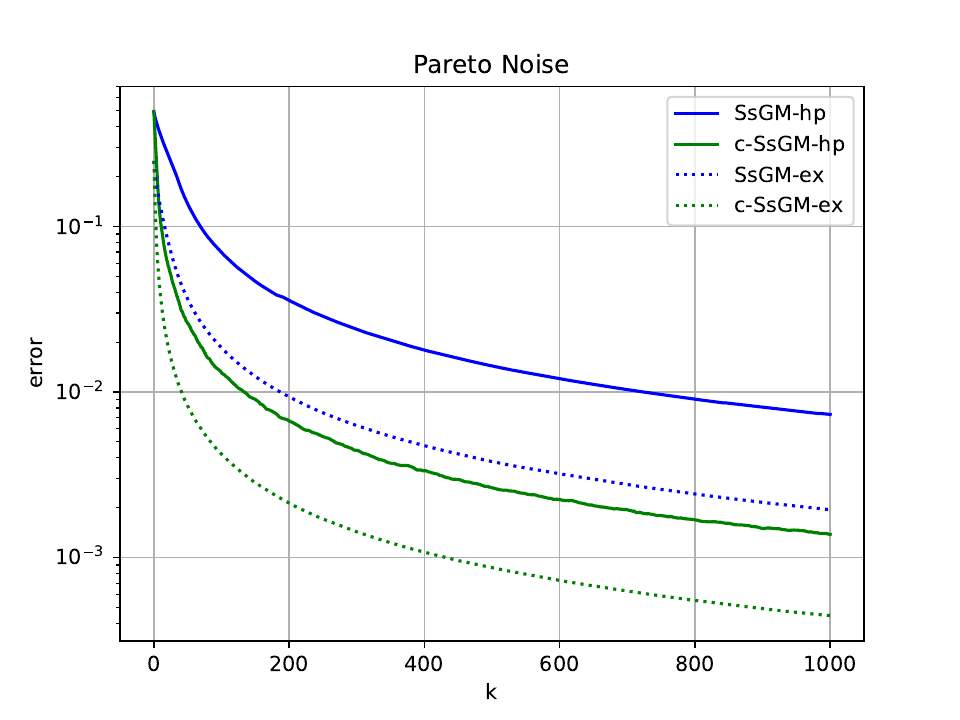}    
\includegraphics[scale=0.3]{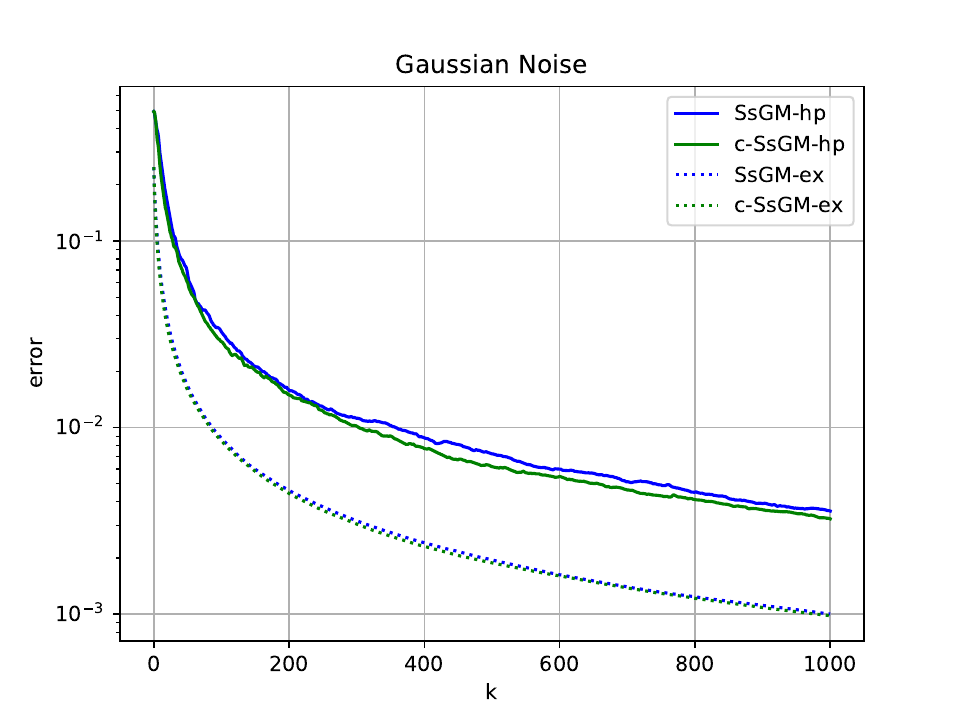}
\includegraphics[scale=0.3]{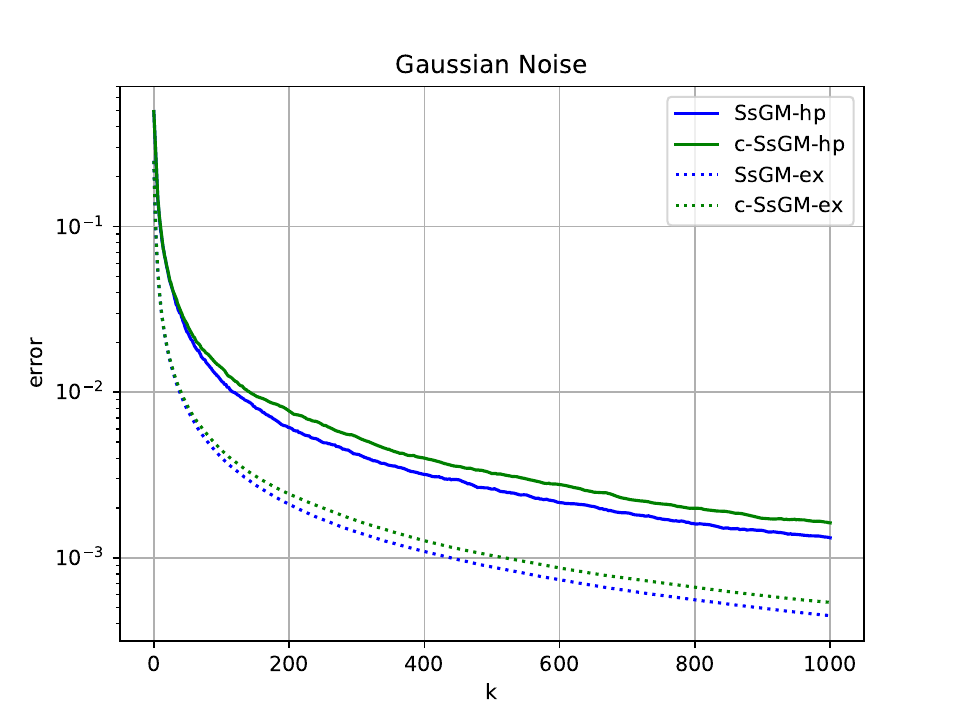}    
\caption{{\rm SsGM} vs {\rm C-SsGM} under different types of noise: Pareto noise {\rm(}left{\rm)}, Gaussian noise with $\sigma= 1$ {\rm(}center{\rm)} and $\sigma=0.1$ {\rm(}right{\rm). The curves show the $99$th percentile and the average of the optimization error over $1000$ repetitions}. 
}
\label{fig:ssgm_fail_new}
\end{figure}

\subsection{Comparison of iteration averaging schemes}\label{exp2}
The goal of this experiment is to compare the performances of the different C-SsGM schemes discussed in Section 2.

We consider the minimization of $f(x)=\|x\|_1$ over the $\ell_2$ ball of radius $1$ in $\mathbb{R}^{d}$, with $d=100$, which contains the global minimum $0$ corresponding to $f^*=0$. Notice that the objective is Lipschitz continuous with $L=10$. At each $x$, the oracle's answer is built by first computing a subgradient of $f$ and then adding a zero mean noise vector $n \in \R^d$ with variance $\sigma^2=100$. The noise vector is generated by independently sampling its components from a Pareto distribution with shape parameter $2.1$ with zero mean and unit variance. We test several average schemes for the iterates with $p=-1/2, p=0$, and $p=1/2$ in the infinite horizon setting. We also consider the finite horizon setting and the coordinate-wise variant of clipping described in \cite{Zhang2020} with $p=0$. We set the step-size schedule as $\gamma_i=\gamma/\sqrt{i}$,  the clipping-level as $\lambda_i = \max\{\beta \sqrt{i},(1+\varepsilon)L\}$, and $\varepsilon=0.001$. We perform a grid search for the parameters $\gamma$ and $\beta$ in $\{0.01,0.02,0.03,0.05,0.06,0.1,0.2,0.3\}$ and $\{0.32,0.64,1.28\}$ respectively. In the case of coordinate-wise clipping, the value of $\beta$ is further divided by the Lipschitz constant $10$. We run $100$ times all the algorithms with batch sizes $m=1$ and $m=\sigma$\footnote{We also test the algorithms with batchsize $m=\sigma^2$ but the results are pretty similar to those with $m=\sigma$.}, till $k=1000$ iterations and measure the $99$th percentile of the optimization error.
\begin{figure}
\centering
\includegraphics[scale=0.40]{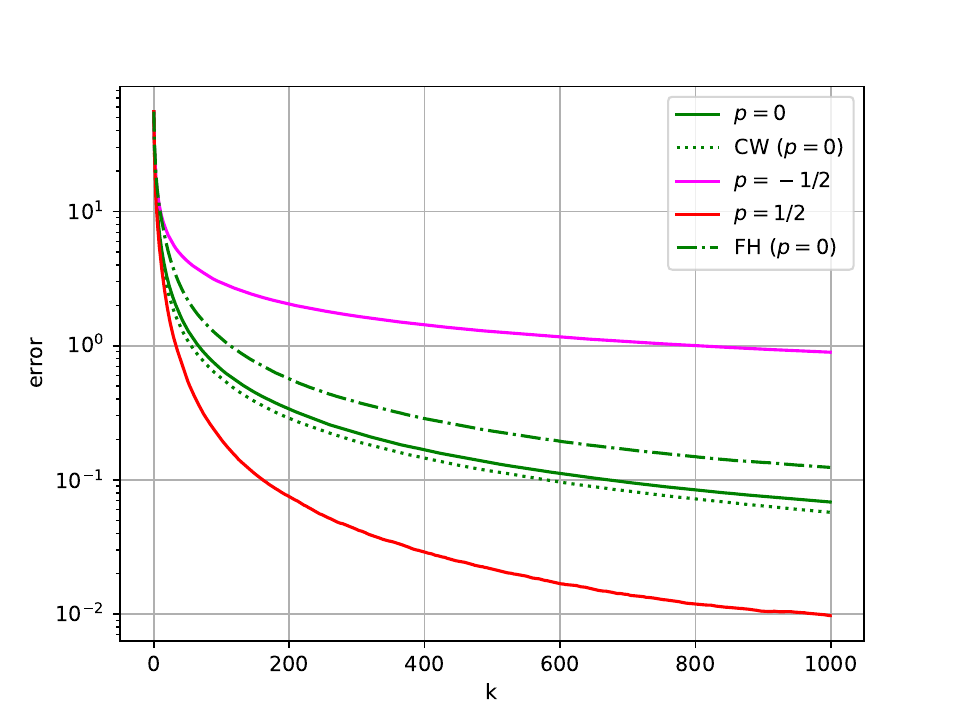}
\includegraphics[scale=0.40]{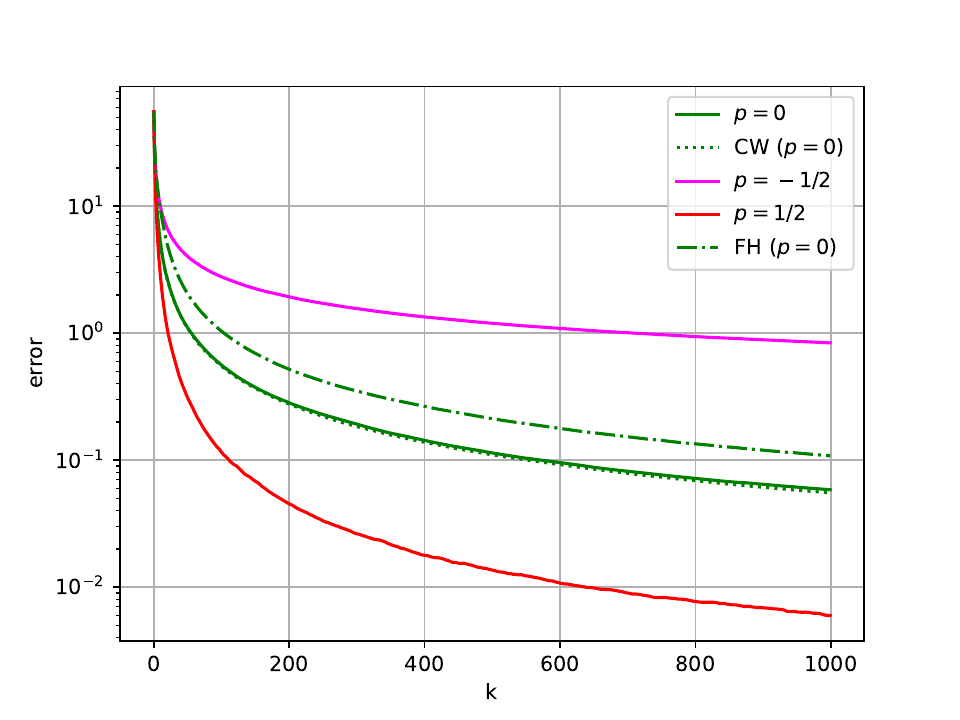}
\caption{Comparison of different averaging schemes for the iterations in infinite and finite-horizon settings with batch sizes $m=1$ (left) and $m=\sigma$ (right). The curves show the $99$th percentile of the optimization error over $100$ repetitions.}
\label{fig:exp2}
\end{figure}
Results are shown in \Cref{fig:exp2} where it is possible to see that: (1) the averaging scheme with $p=1/2$ performs the best, while that with $p=-1/2$ is the worst (see also Remark~\ref{rmk:20240413a}\ref{rmk:20240413a_ii}); (2) the three schemes with $p=0$ performs similarly, with the coordinate-wise variant doing slightly better than the others.

\subsection{Comparison with \cite{Gorbunov2021}} Here we make a comparison with the clipped-SGD algorithm from \cite{Gorbunov2021} on the same problem considered in Section~\ref{exp2}. The algorithm is essentially the same as C-SsGM except for the projection step and the different parameter setting. 
In this respect we note that although in general clipped-SGD cannot handle constraints, it can be used to solve the specific problem at hand, since in this case the constrained and the unconstrained problems share the same solution. Moreover, we consider the finite-horizon scenario with $w_k=1$ which is the same studied in \cite{Gorbunov2021}. For the sake of a fair comparison we run the algorithms with the same batchsize. As for clipped-SGD in \cite{Gorbunov2021}, due to the restrictions on the algorithm's parameters recalled in \eqref{eq:20230830c}-\eqref{eq:20230831a1}-\eqref{eq:20230831a}, by varying $\varepsilon$ one can obtain the following explicit range for the stepsize
\begin{equation}
\label{eq:20230414a}
    \gamma \leq \gamma_{\max}:=D\cdot\min\left\{ 
    \frac{\sqrt{m}}{9 \sigma \sqrt{k \log (4k/\delta)}}, 
    \frac{1}{\sqrt{2k} L},
    \frac{1}{2L \log(4k/\delta)}
    \right\}.
\end{equation}
We perform a grid search on $\gamma$ in $\{\gamma_{\max}/4, \gamma_{\max}/2, \gamma_{\max}\}$. Whereas, for our Clipped-SGM the grid search is done
on the parameters $(\gamma, \beta) \in \{0.1,0.2,0.3\}\times\{0.32, 0.64, 1.28\}$. Note that in our algorithm with finite horizon setting the stepsize is constant and equal to $\gamma/\sqrt{k}$, while the clipping level is possibly varying as $\max\{\beta\sqrt{i}, (1+\varepsilon)L\}=\max\{\beta\sqrt{i}, 10.01\}$. See Theorem~\ref{thm:main2}\ref{thm:main2_iii}.

We report in \Cref{tab:exp3} the $99$th percentile of the optimization error over $100$ repetitions of the algorithms. We also indicate the best parameters chosen via the greed search procedure.
As it is possible to see the parameter setting in \cite{Gorbunov2021} gives tiny values of the stepsizes and quite large values of the clipping levels, which ultimately leads to a noticeable worst performance.

We make a final comment on this comparison. In the numerical section of \cite{Gorbunov2021}, in order to overcome the narrow range of values allowed by the theory, the authors perform a grid-search on the stepsize without the limitation given in \eqref{eq:20230414a}. However, in doing so there is no mathematical guarantee that the selected stepsize leads to a convergent algorithm (according to Theorem~5.1 in \cite{Gorbunov2021}).

\begin{table}[t]
\centering 
\begin{tabular}{|c|c|c|c|c|}
\hline
Method & $m$ & stepsize & clipping level & $99$th-percentile of \\ 
 & & &  & the opt.~error \\[0.1ex] 
\hline
\textsc{C-SsGM} & 
\begin{tabular}{@{}c@{}} 
$1$   \\ 
$10$  \\
$100$ \\ 
\end{tabular} 
& 
\begin{tabular}{@{}c@{}} 
$10^{-2}$ \\ 
$10^{-2}$ \\
$10^{-2}$ \\ 
\end{tabular}
& 
\begin{tabular}{@{}c@{}} 
$10.01$ \\ 
$10.01$ \\
$10.01-20.02$ \\ 
\end{tabular}
&
\begin{tabular}{@{}c@{}} 
$0.124$ \\ 
$0.108$ \\
$0.113$
\end{tabular} 
\\
\hline
\textsc{clipped-SGD}\cite{Gorbunov2021} & 
\begin{tabular}{@{}c@{}} 
$1$  \\ 
$10$ \\
$100$
\end{tabular}  
& 
\begin{tabular}{@{}c@{}} 
$10^{-4}$ \\ 
$3\times 10^{-4}$ \\
$10^{-3}$
\end{tabular}  
&  
\begin{tabular}{@{}c@{}} 
$792.4$ \\ 
$250.6$ \\
$79.2$ \\ 
\end{tabular}
&
\begin{tabular}{@{}c@{}} 
$4.720$ \\ 
$1.570$ \\
$0.463$
\end{tabular} 
\\ 
\hline
\end{tabular}
\vspace{2ex}
\caption{Comparison of C-SsGM with Clipped-SGD from \cite{Gorbunov2021} in the finite-horizon scenario. The parameter $m$ denotes the batchsize. $C-SsGM$
We report the $99$th-percentile of the optimization error over $100$ runs of the experiment.}
\label{tab:exp3}
\end{table}

\subsection{Non-linear classification} We aim at learning a classifier with a low expected classification error when the data generating distribution is a bi-variate standard normal and the target function is the indicator of the first and the third orthants. As a convex surrogate for the classification error we consider the hinge loss $\ell(t,y) \coloneqq \max \{0, 1-ty\}$, where $y \in \{-1,1\}$, and we adopt the Gaussian kernel with scale parameter equal to $10$ to cope with the non-linearity of the problem and let $r=1$; we notice that this results in $L=1$ and $\sigma^2=1$. We run \Cref{algo:kernelSsGM} with $w_k=\sqrt{k}$ for $100$ iterations. The classifier's decision boundary is reported in \Cref{fig:exp3} while its estimated expected classification error, and a test set of $1000$ examples, is about $0.1$.

\begin{figure}
\centering
\includegraphics[scale=0.5]{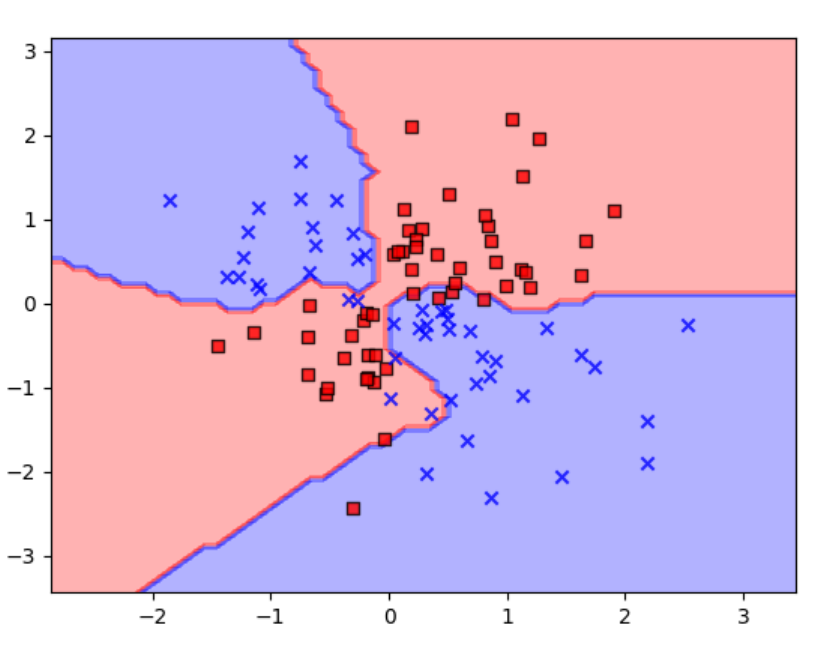}
\caption{Decision boundary of the classifier learned by \Cref{algo:kernelSsGM} with Gaussian kernel along with the negative (red squares) and the positive (blue squares) training examples.}
\label{fig:exp3}
\end{figure}

\section{Conclusion}
In this work we established high probability convergence rates for the projected stochastic subgradient method under heavy-tailed noise, that is, under the sole assumption that the stochastic subgradient oracle has uniformly bounded variance. We provide a unified analysis which simultaneously cover general averaging schemes, stepsizes and clipping levels, while avoiding large numerical constants in the statistical bounds and algorithm parameters. Moreover, we provided an application to the case of statistical learning with kernels which obtains near-optimal performances in a fully practicable fashion. Future interesting directions include avoiding the boundedness assumption on the constraint set and analyzing the last iterates, while keeping the simplicity and practicability of our parameter settings. Yet another valuable future research direction would be to extend our analysis to zero-order optimization settings.

\newpage
\appendix
\appendix
\section{Proofs of Section~\ref{sec:convanal}, Part 1}
\label{appA}

Here we provide detailed proofs of the results in the main paper.

\paragraph{Proof of Lemma~\ref{lemma:CSGproperties}}\ 

\vspace{0.5ex}
\ref{lemma3_i}:
By definition, $\bar{u} = m^{-1} \sum_{\iter=1}^m \hat{u}(x, \xi_j)$, with $(\xi_j)_{1 \leq j \leq m}$ iid random variables and moreover $\Ex[\hat{u}(x, \xi_j)] = u$ and $\Ex[ \norm{\hat{u}(x, \xi_j) - u}^2] \leq \sigma^2$. Thus, we have $\Ex[\bar{u}] = u$ and, for the variance of $\bar{u}$, $\Ex[\norm{\bar{u} - u}^2] \leq \sigma^2/m$. Now, recall that in general we have $\Ex[\norm{\bar{u} - \Ex[\bar{u}]}^2] = \Ex[\norm{\bar{u}}^2] - \norm{\Ex[\bar{u}]}^2$. Therefore, since $\norm{\Ex[\bar{u}]}^2 \leq \Ex[\norm{\bar{u}}^2] \leq L^2$, we have
\begin{equation*}
\Ex[\norm{\bar{u}}^2] =
\Ex[\norm{\bar{u} - \Ex[\bar{u}]}^2] + \norm{\Ex[\bar{u}]}^2 \leq \frac{\sigma^2}{m} + L^2.
\end{equation*}
The first inequality follows by the fact that $\norm{\tilde{u}}^2 \leq \norm{\bar{u}}^2$.
 
\ref{lemma3_ii}:
We start with some preliminaries. Let $\chi \coloneqq \mathbb{I}(\|\bar{u}\| > \lambda)$ and $\eta \coloneqq \mathbb{I}(\norm{\bar{u}-u} > \lambda-L)$, and notice that $\chi \leq \eta$ a.s. Indeed, suppose $\norm{\bar{u}} > \lambda$.
Then, recalling that $\lambda > L$ and $\norm{u} \leq L$, we have
\begin{equation*}
\norm{\bar{u}} \leq \norm{\bar{u}-u} + \norm{u} 
\ \Rightarrow\ \norm{\bar{u}-u} \geq \norm{\bar{u}} - \norm{u} 
> \lambda - L. 
\end{equation*}
Moreover, since $\Ex[\norm{\bar{u} - u}^2] \leq \sigma^2/m$, using Markov's inequality, we have
\begin{equation}
\Ex \eta = \Ex \eta^2 = \Prob\left(\|\bar{u}-u\| > \lambda-L\right) = \Prob\left(\|\bar{u}-u\|^2 > (\lambda-L)^2\right) \leq \frac{\sigma^2}{m(\lambda-L)^2}.
\label{eq:markov}
\end{equation}
Finally, since $\tilde{u} = (1-\chi) \bar{u} + \chi (\lambda\norm{\bar{u}}^{-1} \bar{u}) = \bar{u} + \chi(\lambda\norm{\bar{u}}^{-1}- 1) \bar{u}$, we have
\begin{align*}
\norm{\Ex[\tilde{u}] - u} 
&= \left\|\Ex\left[\bar{u} + \chi \left( \frac{\lambda}{\|\bar{u}\|} - 1\right) \bar{u}\right] - u \right\|
= \left\|\Ex\left[\chi \left( \frac{\lambda}{\|\bar{u}\|} - 1\right) \bar{u} \right] \right\|\\[2ex]
&\leq \Ex\left\|\chi \left( \frac{\lambda}{\|\bar{u}\|} - 1\right) \bar{u} \right\| 
\leq \Ex\bigg[\chi \left(1 - \frac{\lambda}{\|\bar{u}\|}\right) \norm{\bar{u}} \bigg]\\[2ex] 
& \leq \Ex\big[\chi \left(\norm{\bar{u}} - \lambda \right) \big]
\leq \Ex\big[\chi (\norm{\bar{u}} - \norm{u}) \big]\\[2ex]
&\leq \Ex\big[\chi \norm{\bar{u}-u} \big]
\leq \sqrt{\Ex[\chi^2] \Ex[\norm{\bar{u} - u}^2]} \\[2ex]
&\leq \sqrt{\Ex\eta^2 \Ex\norm{\bar{u}-u}^2} 
\leq \sqrt{ \frac{\sigma^2}{m(\lambda-L)^2} \frac{\sigma^2}{m}} \\
&\leq  \frac{\sigma^2}{m(\lambda-L)}.
\end{align*}

\ref{lemma3_iv}:
We have $\tilde{u} - u = (1-\chi)(\tilde{u} - u)+ \chi (\tilde{u} - u) = (1-\chi)(\bar{u} - u)+ \chi (\lambda \norm{\bar{u}}^{-1} \bar{u} - u)$, and hence
\begin{align*}
\Ex\norm{\tilde{u} - u}^2 &= 
\Ex\norm{(1-\chi)(\bar{u} - u)+ \chi (\lambda \norm{\bar{u}}^{-1} \bar{u} - u)}^2 \\[1ex]
&\leq \Ex[ (1-\chi)\norm{\bar{u}-u}^2] + \Ex [\chi\norm{\lambda\norm{\bar{u}}^{-1}\bar{u} - u}^2] \\[1ex]
&\leq  \Ex\norm{\bar{u}-u}^2 
+ \Ex [\chi(\lambda + L)^2] \\[1ex]
&\leq \frac{\sigma^2}{m} +
(\lambda + L)^2 \Ex[\eta]\\[1ex]
& \leq \frac{\sigma^2}{m}\bigg[1 + \bigg(\frac{\lambda+L}{\lambda-L}\bigg)^2 \bigg].
\end{align*}
\ref{lemma3_iii}:
This follows from \ref{lemma3_iv}
and the standard
fact that $\Ex \|X - \Ex[X] \|^2 \leq \Ex \|X - a\|^2$ for any vector $a$.
\qed

\paragraph{Proof of Lemma~\ref{lemma:ErrorDecomp} {\rm(}Decomposition of the error{\rm)}}\ 

\vspace{0.5ex}
First, by the definition of $x_{i+1}$, we derive
\begin{align*}
\|x_{\iter+1}-x\|^2 &\leq \|x_{\iter} - \gamma_\iter \tilde{u}_i -x\|^2 = \|x_{\iter} - x \|^2 - 2 \gamma_i \langle x_{\iter} - x, \tilde{u}_i \rangle + \gamma_\iter^2 \|\tilde{u}_\iter\|^2 \\
&= \|x_{\iter} - x \|^2 - 2 \gamma_\iter \langle x_{\iter} - x, u_\iter + (\tilde{u}_\iter - u_\iter) \rangle + \gamma_\iter^2 \|\tilde{u}_\iter\|^2 \\
&= \|x_{\iter} - x \|^2 - 2 \gamma_\iter \langle x_{\iter} - x, \tilde{u}_\iter - u_\iter \rangle - 2 \gamma_\iter \langle x_{\iter} - x, u_\iter \rangle + \gamma_\iter^2 \|\tilde{u}_\iter\|^2
\end{align*}
and hence
\begin{equation}
\label{eq:20220510a}
2 \gamma_\iter \langle x_{\iter} - x, u_\iter \rangle \leq \|x_{\iter} - x \|^2 - \|x_{\iter+1} - x \|^2 - 2 \gamma_\iter \langle x_{\iter} - x, \tilde{u}_\iter - u_\iter \rangle + \gamma_\iter^2 \|\tilde{u}_\iter\|^2.
\end{equation}
Now, since $u_\iter \in \partial f(x_\iter)$,
we have $f(x_\iter) - f(x) \leq \langle x_{\iter} - x, u_\iter \rangle$, which, together with \eqref{eq:20220510a}, yield
\begin{align*}
w_\iter(f(x_\iter) - f(x)) &\leq \frac{w_\iter}{2\gamma_\iter}
(\norm{x_{t} - x}^2 - \norm{x_{t+1} - x}^2)
+ w_\iter\langle x - x_{t}, \tilde{u}_\iter - u_\iter \rangle 
+ \frac{w_\iter\gamma_\iter}{2} \norm{\tilde{u}_\iter}^2.
\end{align*}
Summing up all the inequalities, dividing by $\sum_{\iter=1}^k w_\iter$, evaluating at $x=x^* \in \argmin_X f$ and using Jensen's inequality we get
\begin{multline}
\label{eq:20220512b}
 f(\bar{x}_\Iter) - f(x^*) \leq \frac{1}{\sum_{\iter=1}^k w_\iter} 
\bigg[ \frac 1 2\sum_{\iter=1}^{\Iter}
\frac{w_\iter}{\gamma_\iter} 
(\norm{x_{\iter} - x^*}^2 - \norm{x_{\iter+1} - x^*}^2)\\
+ \sum_{\iter=1}^{\Iter} w_\iter \langle x^* - x_{t}, \tilde{u}_\iter - u_\iter \rangle + \frac{1}{2}
\sum_{\iter=1}^{\Iter} w_\iter \gamma_\iter\norm{\tilde{u}_\iter}^2\bigg]
\end{multline}
Now, if we consider the first term within the brackets of \eqref{eq:20220512b}, we
get
\begin{align}
\label{eq:20240414b}
\nonumber\frac 1 2\sum_{\iter=1}^{\Iter} \frac{w_\iter}{\gamma_\iter} (\norm{x_{\iter} - x^*}^2 - \norm{x_{\iter+1} - x^*}^2)
& = \frac 1 2 \bigg[  \sum_{\iter=1}^\Iter \frac{w_\iter}{\gamma_\iter}
\norm{x_{\iter} - x^*}^2 - \sum_{\iter=1}^\Iter \frac{w_\iter}{\gamma_\iter} \norm{x_{\iter+1} - x^*}^2\bigg]\\
\nonumber& = \frac 1 2 \bigg( \frac{w_1}{\gamma_1} \norm{x_{1} - x^*}^2 - \frac{w_\Iter}{\gamma_\Iter} \norm{x_{\Iter+1} - x^*}^2 \bigg)\\
\nonumber&\qquad\qquad + \frac 1 2  \sum_{\iter=1}^{\Iter-1}\bigg(\frac{w_{\iter+1}}{\gamma_{\iter+1}}- \frac{w_{\iter}}{\gamma_{\iter}}\bigg) \norm{x_{\iter+1} - x^*}^2\\
\nonumber& \leq \frac 1 2 \bigg( \frac{w_1}{\gamma_1} D^2 - \frac{w_\Iter}{\gamma_\Iter} \norm{x_{\Iter+1} - x^*}^2 \bigg)\\
\nonumber&\qquad\qquad + \frac{D^2}{2}
\sum_{\iter=1}^{\Iter-1}\bigg(\frac{w_{\iter+1}}{\gamma_{\iter+1}}- \frac{w_{\iter}}{\gamma_{\iter}}\bigg)\\
\nonumber& \leq \bigg( \frac{w_1}{\gamma_1} \frac{D^2}{2} - \frac{w_k}{\gamma_\Iter} \frac{\norm{x_{\Iter+1} - x^*}^2}{2} \bigg)\\
\nonumber&\qquad\qquad +
\frac{D^2}{2}\bigg(\frac{w_\Iter}{\gamma_\Iter}- \frac{w_1}{\gamma_1}\bigg)\\
& \leq \frac{D^2}{2} \frac{w_\Iter}{\gamma_\Iter}.
\end{align}
Concerning, the second term in the RHS of \eqref{eq:20220512b},
setting $\tilde{\mu}_\iter = \Ex[\tilde{u}_\iter\,\vert\, x_1, \dots, x_\iter]$, we have
\begin{align*}
\sum_{\iter=1}^\Iter w_\iter 
\langle x^* - x_{\iter}, \tilde{u}_\iter - u_\iter \rangle
&= \sum_{\iter=1}^\Iter w_\iter 
\langle x^* - x_{\iter}, \tilde{u}_\iter - \tilde{\mu}_\iter \rangle
+ \sum_{\iter=1}^\Iter w_\iter 
\langle x^* - x_{\iter}, \tilde{\mu}_\iter - u_\iter \rangle\\
&= \sum_{\iter=1}^\Iter \theta_\iter^v + \sum_{\iter=1}^\Iter \theta_\iter^b.
\end{align*}
Finally, considering the last term in \eqref{eq:20220512b}, we have
\begin{align*}
\frac{1}{2} \sum_{\iter=1}^{\Iter} w_\iter\gamma_\iter\norm{\tilde{u}_\iter}^2 
&= \frac{1}{2} \sum_{\iter=1}^{\Iter} 
w_\iter\gamma_\iter \big(\norm{\tilde{u}_\iter}^2 - \Ex\big[\norm{\tilde{u}_\iter}^2\,\vert\, x_1, \dots, x_\iter\big]\big)\\
&\qquad\qquad + \frac{1}{2} \sum_{\iter=1}^{\Iter} w_\iter\gamma_\iter\Ex\big[\norm{\tilde{u}_\iter }^2\,\vert\, x_1, \dots, x_\iter\big]\\
&= \frac{1}{2}\sum_{\iter=1}^{\Iter} \zeta_\iter + \frac{1}{2} \sum_{\iter=1}^{\Iter} \nu_\iter.
\end{align*}
The statement follows. \qed

\paragraph{Proof of Proposition~\ref{prop:FreedmanBound} {\rm(}Freedman's Bound{\rm})}\ 

\vspace{0.5ex}
In the Bernstein's inequality for martingale (Fact~\ref{thm:freedman}) we have
\begin{equation*}
\Prob \left( \sum_{\iter=1}^{\Iter} X_\iter > \eta, \sum_{\iter=1}^{\Iter} \sigma_\iter^2 \leq F \right) \leq \exp{\left(-\frac{\eta^2}{2(F+\frac{\eta c}{3})}\right)}.
\end{equation*}
Now, for every $\eta>0$ and $\delta>0$ we have
\begin{equation*}
    \exp{\left(-\frac{\eta^2}{2(F+\frac{\eta c}{3})}\right)}
    \leq \frac{\delta}{2}\ \Leftrightarrow\ 
    \frac{\eta^2}{2(F+\frac{\eta c}{3})} \geq \log\left(\frac{2}{\delta}\right)\ \Leftrightarrow\ \eta^2 - \frac{2}{3} c \log\left(\frac{2}{\delta}\right) \eta - 2 F \log\left(\frac{2}{\delta}\right) \geq 0.
\end{equation*}
Hence, since we are assuming that  $\eta > 0$, it follows that 
the last inequality is equivalent to
\begin{equation*}
\eta \geq \frac{1}{3} c \log\left(\frac{2}{\delta}\right) + \frac{1}{2}\sqrt{\frac{4}{9} c^2 \log^2\left(\frac{2}{\delta}\right) + 8F\log\left(\frac{2}{\delta}\right)} \;.
\end{equation*}
Therefore, since $\sum_{i=1}^k \sigma_i^2 \leq F$ a.s., we have
\begin{equation*}
    \Prob\bigg(\sum_{i=1}^k X_i > \frac{1}{3} c \log\left(2/\delta\right) + \frac{1}{2}\sqrt{\frac{4}{9} c^2 \log^2\left(2/\delta\right) + 8F\log\left(2/\delta\right)} \bigg) \leq \frac \delta 2
\end{equation*}
and hence
\begin{equation*}
    \Prob\bigg(\sum_{i=1}^k X_i \leq \frac{1}{3} c \log\left(2/\delta\right) + \frac{1}{2}\sqrt{\frac{4}{9} c^2 \log^2\left(2/\delta\right) + 8F\log\left(2/\delta\right)} \bigg) \geq 1- \frac \delta 2.
\end{equation*}
Now the statement follows by noting that
\begin{equation*}
   \sqrt{\frac{4}{9} c^2 \log^2\left(\frac{2}{\delta}\right) + 8F\log\left(\frac{2}{\delta}\right)} 
   \leq 
   \frac{2}{3} c \log\left(\frac{2}{\delta}\right) + 2\sqrt{2F\log\left(\frac{2}{\delta}\right)}.
\end{equation*}
\vspace{-1ex}
\qed

\paragraph{Proof of Proposition~\ref{prop:AnalysisA} {\rm(}Analysis of the term {\bf A}{\rm)}}\ 

\vspace{0.5ex}
Let $i \in \{1,\dots, k\}$ and set $\tilde{\mu}_i = \Ex[ \tilde{u}_i\,\vert\, x_1, \dots, x_i]$. Then
we have
\begin{equation}
\abs{\theta_\iter^v} = 
w_\iter \abs{\langle \tilde{u}_\iter - \tilde{\mu}_\iter, x - x_\iter \rangle }
\leq w_\iter\norm{\tilde{u}_\iter -\! \tilde{\mu}_\iter} \norm{x \!- x_\iter}
\!\leq\! 2 D w_\iter\lambda_\iter
\leq 2 D \max_{1 \leq \iter \leq \Iter} w_\iter \lambda_\iter \eqqcolon c \;.
\end{equation}
Moreover, 
it is clear that
\begin{equation*}
\Ex\big[w_\iter\langle \tilde{u}_\iter - \tilde{\mu}_\iter, x - x_\iter \rangle\,\big\vert\, x_1, \dots, x_\iter\big] 
= w_\iter\langle \Ex\big[\tilde{u}_\iter - \tilde{\mu}_\iter\,\big\vert\, x_1, \dots, x_\iter\big], x - x_\iter \rangle = 0
\end{equation*}
and by the variance bound in Lemma~\ref{lemma:CSGproperties}\ref{lemma3_iv},
we have
\begin{align*}
\sigma_\iter^2 \coloneqq \Ex[ w_\iter^2\langle \tilde{u}_\iter - \tilde{\mu}_\iter, x^* - x_\iter \rangle^2\,\vert\, x_1,\dots,x_\iter] 
&\leq 
w_\iter^2\Ex[ \norm{\tilde{u}_\iter - \tilde{\mu}_\iter}^2 \norm{x^* - x_\iter}^2\,\vert\, x_1,\dots, x_\iter]\\
&\leq
\dfrac{D^2\sigma^2}{m} w_\iter^2\bigg[1 + \bigg(\dfrac{\lambda_\iter+L}{\lambda_\iter-L}\bigg)^2 \bigg]\\
& \leq \dfrac{D^2\sigma^2}{m} w_\iter^2 \bigg[1 + \left(\frac{2+\epsilon}{\epsilon}\right)^2 \bigg]\\
& \leq 4\dfrac{D^2\sigma^2}{m} w_\iter^2 \bigg(1+\frac 1 \varepsilon\bigg)^2
\end{align*}
Hence, the total conditional variance is
\begin{equation*}
V_\Iter:=\sum_{\iter=1}^{\Iter} \sigma_\iter^2 \leq 
4\dfrac{D^2\sigma^2}{m} \bigg(1+\frac 1 \varepsilon\bigg)^2 \sum_{i=1}^k w_\iter^2.
\end{equation*}
We now define
\begin{equation*}
F = 4\frac{D^2 \sigma^2}{m} \bigg(1+\frac 1 \varepsilon\bigg)^2\sum_{\iter=1}^\Iter 
w_\iter^2 \;.
\end{equation*}
Then by \Cref{prop:FreedmanBound} we have that with probability at least $1-\frac{\delta}{2}$
\begin{align*}
\sum_{\iter=1}^{\Iter} \theta_\iter^v &\leq  \frac{2}{3} c \log\left(\frac{1}{\delta}\right) + \sqrt{ 2 F\log\left(\frac{1}{\delta}\right)} \nonumber \\
&= \frac{4}{3} D \cdot \max_{1 \leq \iter \leq \Iter} w_\iter \lambda_\iter \cdot \log\left(\frac{2}{\delta}\right) + 2 \bigg(1+\frac 1 \varepsilon\bigg) \frac{D \sigma}{\sqrt{m}}\sqrt{2\sum_{\iter=1}^\Iter 
w_\iter^2  \cdot \log\left(\frac{2}{\delta}\right)} 
\end{align*}
and the statement follows. \qed

\paragraph{Proof of Proposition~\ref{prop:AnalysisB} {\rm(}Analysis of the term {\bf B}{\rm)}}\ 

\vspace{0.5ex}
Let $i \in \{1,\dots, k\}$ and set
$\tilde{\mu}_\iter=\Ex[\tilde{u}_i \,\vert\, x_1, \dots, x_i]$. Then
by the bias bound in Lemma~\ref{lemma:CSGproperties}\ref{lemma3_ii}
and Remark~\ref{rmk:eps}, we have
\begin{equation*}
\theta_\iter^b = w_\iter\langle \tilde{\mu}_\iter - u_\iter, x - x_\iter \rangle \leq 
w_\iter \norm{\tilde{\mu}_\iter - u_\iter} \norm{x-x_\iter} \leq 
\frac{D \sigma^2}{m} w_\iter 
\frac{1}{(\lambda_\iter-L)} 
\leq  \frac{D \sigma^2}{m} \bigg(1+\frac{1}{\varepsilon}\bigg) \frac{w_\iter}{\lambda_\iter} \quad \text{a.s.}
\end{equation*} 
Thus the statement follow.\qed

\paragraph{Proof of Proposition~\ref{prop:AnalysisC} {\rm(}Analysis of the term {\bf C}{\rm)}}\ 

\vspace{0.5ex}
Set, for the sake of brevity, $\tilde{\xi}_\iter = \Ex[\norm{\tilde{u}_\iter}^2 \,\vert\, x_1, \dots, x_i]$.
Then, we have
\begin{equation*}
    \abs{\zeta_\iter} = w_\iter \gamma_\iter 
    \abs{\norm{\tilde{u}_\iter}^2 - \tilde{\xi}_\iter}
    \leq 2 w_\iter \gamma_\iter \lambda_\iter^2 
    \leq 2 \max_{1 \leq \iter \leq \Iter} w_\iter \gamma_\iter \lambda^2_\iter =: c
\end{equation*}
and clearly 
\begin{equation*}
   \Ex[\zeta_\iter\,\vert\, x_1, \dots, x_i] = w_\iter \gamma_\iter \Ex[\norm{\tilde{u}_\iter}^2 - \tilde{\xi}_\iter \,\vert\, x_1, \dots, x_i]=0.
\end{equation*}
Moreover,
\begin{align*}
    \sigma^2_\iter &= w_\iter^2 \gamma_\iter^2 
    \Ex[(\norm{\tilde{u}_\iter}^2 - \tilde{\xi}_\iter)^2\,\vert\, x_1, \dots, x_i] \\[1ex]
    &\leq w_\iter^2 \gamma_\iter^2 \Ex[\norm{\tilde{u}_\iter}^4\,\vert\, x_1, \dots, x_i]\\[1ex]
    &= w_\iter^2 \gamma_\iter^2  \Ex[\norm{\tilde{u}_\iter}^2 \norm{\tilde{u}_\iter}^2\,\vert\, x_1, \dots, x_i]\\[1ex]
    & \leq w_\iter^2 \gamma_\iter^2 \lambda_\iter^2 
    \Ex[\norm{\tilde{u}_\iter}^2\,\vert\, x_1, \dots, x_i]\\[1ex]
    & \leq w_\iter^2 \gamma_\iter^2 \lambda_\iter^2 \bigg(\frac{\sigma^2}{m} + L^2\bigg).
\end{align*}
Thus, 

\begin{equation*}
V_{\Iter} \coloneqq \sum_{\iter=1}^\Iter \sigma^2_\iter \leq \bigg(\frac{\sigma^2}{m} + L^2 \bigg) \sum_{\iter=1}^\Iter w_\iter^2 \gamma_\iter^2 \lambda_\iter^2.
\end{equation*}
Now, setting

\begin{equation*}
F = \bigg(\frac{\sigma^2}{m} + L^2 \bigg) \sum_{\iter=1}^\Iter w_\iter^2 \gamma_\iter^2 \lambda_\iter^2 \;,
\end{equation*}
by \Cref{prop:FreedmanBound} we have that with probability at least $1-\delta$

\begin{align}
\sum_{i=1}^{\Iter} \zeta_\iter &\leq \eta \coloneqq \frac{2}{3} c \log\left(\frac{1}{\delta}\right) + \sqrt{2} \sqrt{F\log\left(\frac{1}{\delta}\right)} \nonumber \\
&= \frac{4}{3} \cdot \max_{1 \leq \iter \leq \Iter} w_\iter \gamma_\iter \lambda^2_\iter \cdot \log\left(\frac{1}{\delta}\right) + \sqrt{2} \sqrt{\bigg(\frac{\sigma^2}{m} + L^2 \bigg) \sum_{\iter=1}^\Iter w_\iter^2 \gamma_\iter^2 \lambda_\iter^2 \cdot \log\left(\frac{1}{\delta}\right)} \;.
\label{eq:Cbound}
\end{align}
\qed

\vspace{1ex}
\section{Proofs of Section~\ref{sec:convanal}, Part 2}
\label{appB}

Here we provide the proofs of the results pertaining to the part 2 of Section~\ref{sec:convanal}.

\paragraph{Proof of Lemma~\ref{lem:harmonicsums}}\  

\vspace{0.5ex}
We first prove the upper bounds. We start from the case $p \geq 0$. Since $\R_{++} \ni x \mapsto x^p$ is increasing, we have

\begin{align*}
\sum_{i=1}^k i^p &= \sum_{i=1}^{k-1} i^p + k^p \leq \int_{1}^{k} x^p dx + k^p = \left [\frac{x^{p+1}}{p+1} \right ]_{x=1}^{x=k} + k^p = \frac{k^{p+1}-1}{p+1} + k^p \\
&= \frac{k^{p+1}}{p+1} + k^p - \frac{1}{p+1} \;.
\end{align*}
Now
\begin{equation*}
\frac{k^{p+1}}{p+1} + k^p - \frac{1}{p+1} \leq k^{p+1}\ 
\Leftrightarrow\ 
\frac{p}{p+1} k^{p+1} - k^p \geq - \frac{1}{p+1}.
\end{equation*}
Since $f(x) = \frac{p}{p+1} x^{p+1} - x^p$ is increasing in $[1, \infty)$, the minimum is achieved at $x=1$. As a result, the above inequality is equivalent to
\begin{align*}
\frac{p}{p+1} - 1 \geq - \frac{1}{p+1}\Leftrightarrow\ \frac{p}{p+1} \geq 1-\frac{1}{p+1},
\end{align*}
which is always true. We have proved that for all $p \geq 0$
\begin{equation}
\label{eq:lemma3.9-1}
\sum_{i=1}^k i^p \leq k^{p+1} \;.
\end{equation}
Now consider $p<0$. Since $\R_{++} \ni x \mapsto x^p$ is decreasing, we have
\begin{align*}
\label{eq:lemma3.9-2}
\nonumber\sum_{i=1}^k i^p &= 1 + \sum_{i=2}^{k} i^p \leq 1 + \int_{1}^{k} x^p dx = 
\begin{cases}
\frac{k^{p+1}}{p+1} + \frac{p}{p+1} &\text{if } p \neq -1\\[1ex]
1 + \log k &\text{if } p = -1
\end{cases}\\[4ex]
&\leq
\begin{cases}
\frac{k^{p+1}}{p+1}  &\text{if } -1<p<0\\[1ex]
1 + \log k &\text{if } p = -1\\
\frac{p}{p+1} &\text{if } p < -1
\end{cases}
\end{align*}
and the bounds from above follow.
We now address the lower bounds. Suppose $p \geq 0$, then $f(x) = x^p$ is increasing and 
\begin{equation}
\label{eq:lemma3.9-3}
\sum_{i=1}^k i^p = 1 + \sum_{i=2}^k i^p \geq 1 + \int_{1}^k x^p dx = 1 + \frac{k^{p+1}-1}{p+1}
\geq \frac{k^{p+1}}{p+1}.
\end{equation}
Now suppose that $p \in \left[-1,0\right[$. Let $f(x) = x^p$, and consider the primitive 
\begin{equation*}
F(x) \coloneqq 
\begin{cases}
\frac{x^{p+1}}{p+1}, \quad \text{if } p \neq -1, \\[1ex]
\log x, \quad \text{if } p = -1.
\end{cases}
\end{equation*}
Then since $f(x)$ is convex and decreasing, from the trapezoidal rule we have
\begin{align*}
\sum_{i=1}^k i^p \geq F(k) - F(1) + \frac{1}{2}(f(1)+f(k)) \eqqcolon (\star) \;.
\end{align*}
Suppose $p \neq -1$. Then,
\begin{align*}
(\star) = \frac{1}{p+1}(k^{p+1}-1) + \frac 1 2 (1 + k^p) = \frac{k^{p+1}}{p+1} + \frac{k^p}{2} + \frac 1 2 - \frac{1}{p+1}.
\end{align*}
Now note that
\begin{align*}
\frac{k^{p+1}}{p+1} + \frac{k^p}{2} + \frac{1}{2} - \frac{1}{p+1} \geq k^{p+1}\ \Leftrightarrow\ \left(\frac{1}{p+1}-1\right) k^{p+1} + \frac{k^{p}}{2} \geq \frac{1}{p+1} -  \frac{1}{2} \;. 
\end{align*}
Moreover the function $g(x) = (-p/(p+1)) x^{p+1} + x^{p}/2$ is increasing in $\left[1/2, +\infty\right[$ and its minimum value over the integers is attained at $x=1$. Thus the above inequality is equivalent to
\begin{align*}
\left(\frac{1}{p+1} -1\right) + \frac 1 2 \geq \frac{1}{p+1} -  \frac{1}{2} \;,
\end{align*}
which is trivially true. Thus we have showed that
\begin{equation*}
\sum_{i=1}^k i^p \geq k^{1+p}.
\end{equation*}
Finally, suppose that $p=-1$. Then
\begin{equation*}
    (\star) = \log(k) - \log(1) + \frac{1}{2} \left( 1 + \frac{1}{k} \right) = \log(k) + \frac{1}{2} + \frac{1}{2k} \geq \log(k) + \frac{1}{2},
\end{equation*}
which concludes the proof.\qed

\paragraph{Proof of Lemma~\ref{lem:maxseq}}
Set for the sake of brevity $a_i = \max\{b\,i^t, c\,i^s\}$.
Suppose first that $t \geq 0$ and $s \leq 0$. Note that if $t - s = 0$, then $t=s=0$ and the statement holds. Therefore we assume $t - s \neq 0$. Now note that 
\begin{equation*}
b\, i^t \geq c\, i^s\ \Leftrightarrow\ i^{t-s} \geq \frac{c}{b}\ \Leftrightarrow\ i \geq \left(\frac{c}{b}\right)^{\frac{1}{t-s}}.
\end{equation*}
Let $n \coloneqq \big\lceil (c/b)^{\frac{1}{t-s}}  \big\rceil$, so that $i \geq n\ \Leftrightarrow\ i\geq (c/b)^{\frac{1}{t-s}}$. Consider the following cases on $k$.
\begin{itemize}
\item $k < n$. Then, $\forall\, i = 1,\dots,k$, $i < (c/b)^{\frac{1}{t-s}}$. Thus 
\begin{equation*}
\forall\, i = 1,\dots,k,\quad    a_i=\max\{b\, i^t, c\, i^s\} = c\, i^s\ \Rightarrow\ \max_{1 \leq i \leq k} a_i = c.
\end{equation*}
\item $k \geq n$. Then, $\forall\, i = 1, \dots,n-1, i < (c/b)^{\frac{1}{t-s}}$, and hence $\max\{b\, i^t, c\, i^s\} = c\, i^s\ \Rightarrow\ \max_{1 \leq i \leq n-1} a_i = c$. 
Moreover, $\forall\, i = n, \dots,k, i \geq (c/b)^{\frac{1}{t-s}}$, and hence $\max\{b\, i^t, c\, i^s\} = b\, i^t \Rightarrow \max_{n \leq i \leq k}  a_i = b\, k^t$. Thus, 
\begin{equation*}
\max_{1 \leq i \leq k} a_i = \max\{b\, k^t, c\} = \max\{b\, k^{t_+}, c\, k^{s_+}\} \;.
\end{equation*}
\end{itemize}

Now we consider the other cases. If $t\geq 0$ and $s \geq 0$,
or $t\leq 0$ and $s \leq 0$, then clearly the statement follows since the sequences $i \mapsto i^t$ and $i \mapsto i^s$ are both increasing or both decreasing. If $t \leq 0$ and $s \geq 0$,
then we can write $a_i = \max\{c\, i^s, b\, i^t\}$ and we can apply the first case obtaining $\max_{1 \leq i \leq k} a_i = \max\{c\,k^s, b\} = \max\{b\,k^{t_+}, c\, k^{s_+}\}$. The proof is complete. \qed

\vspace{1ex}
\paragraph{Proof of Lemma~\ref{lem:summax}}
Let $a_i = \max\{b\,i^t, c\,i^s\}$,
 $n = \lceil (c/b)^{1/(t-s)}\rceil$ and set $M = (c/b)^{1/(t-s)}$.
Then
\begin{equation*}
    i \geq n\ \Leftrightarrow\ i \geq (c/b)^{\frac{1}{t-s}}\ \Leftrightarrow\  b\,i^t \geq c\,i^s \ \Leftrightarrow\ a_i = b\,i^t.
\end{equation*}
We consider two cases as before.
\begin{itemize}
    \item $k<n$. Then, for every $i=1,\dots, k$, we have $a_i = c\,i^s$ and hence, by \Cref{lem:harmonicsums},
    \begin{equation}
    \label{eq:20230830b}
        \sum_{i=1}^k a_i = c \sum_{i=1}^k i^s \leq c\cdot
        \begin{cases}
        \displaystyle
        \frac{k^{s+1}}{(s+1)\textcolor{blue}{\wedge} 1} &\text{if } -1<s\\[2ex]
        1 + \log k &\text{if } -1=s\\[1ex]
        \displaystyle
        \frac{s}{s+1} &\text{if } s<-1
        \end{cases}
    \end{equation}
    \item $k \geq n$. Then, if $n>1$, for every $i=1, \dots, n-1$, we have $a_i = c i^s$ and hence (again by \Cref{lem:harmonicsums})
    \begin{equation*}
        \sum_{i=1}^{n-1} a_i = c \sum_{i=1}^{n-1} i^s \leq c\cdot
        \begin{cases}
        \displaystyle
        \frac{(n-1)^{s+1}}{(s+1)\textcolor{blue}{\wedge} 1} &\text{if } -1<s\\[2ex]
        1 + \log (n-1) &\text{if } -1=s\\[1ex]
        \displaystyle
        \frac{s}{s+1} &\text{if } s<-1.
        \end{cases}
    \end{equation*}
    Whereas, for every $i=n, \dots, k$, $a_i = b\, i^t$ and hence
    \begin{equation*}
        \sum_{i=n}^{k} a_i \leq b \sum_{i=1}^{k} i^s \leq b\cdot
        \begin{cases}
        \displaystyle
        \frac{k^{t+1}}{(t+1)\textcolor{blue}{\wedge} 1} &\text{if } -1<t\\[2ex]
        1 + \log k &\text{if } -1=t\\[1ex]
        \displaystyle
        \frac{t}{t+1} &\text{if } t<-1.
        \end{cases}
    \end{equation*}
    Therefore (for $k \geq n$), 
    \begin{equation}
    \label{eq:20230830a}
        \sum_{i=1}^k a_i \leq
        \begin{cases}
        \displaystyle
            \frac{c(n-1)^{s+1}}{(s+1)\wedge 1} +\frac{b\,k^{t+1}}{(1+t)\wedge 1} &\text{if } -1 < s <t\\[2ex]
            \displaystyle
            c(1 + \log (n-1))_+ +\frac{b\,k^{t+1}}{(1+t)\wedge 1} &\text{if } -1 = s <t\\[2ex]
            \displaystyle
            \frac{c s}{s+1} +\frac{b\,k^{t+1}}{(1+t)\wedge 1} &\text{if } s<-1 <t\\[2ex]
            \displaystyle
            \frac{c s}{s+1} + b(1+\log k) &\text{if } s<-1 = t\\[2ex]
            \displaystyle
            \frac{c s}{s+1} + \frac{b t}{t+1} &\text{if } s<t<-1.
        \end{cases}
    \end{equation}
\end{itemize}
We are now ready to bound the sum
\begin{equation*}
    \sum_{i=1}^k a_i.
\end{equation*}
We consider different cases.\\
Case $-1<s<t$. Suppose $k \geq (c/b)^{\frac{1}{t-s}}$. Then $n-1 \leq (c/b)^{\frac{1}{t-s}} \leq k$ and hence
\begin{equation*}
c(n-1)^{s+1} \leq c \bigg(\frac{c}{b}\bigg)^{\frac{s+1}{t-s}} = b \frac c b \bigg(\frac{c}{b}\bigg)^{\frac{s+1}{t-s}} = b \bigg(\frac{c}{b}\bigg)^{\frac{t+1}{t-s}} \leq b k^{t+1}.
\end{equation*}
Therefore, recalling the first of \eqref{eq:20230830a}, we have
\begin{equation*}
\sum_{i=1}^k a_i \leq \bigg(\frac{1}{(s+1)\wedge 1} + \frac{1}{(1+t)\wedge 1} \bigg) b k^{t+1}.
\end{equation*}
In the end, recalling also \eqref{eq:20230830b} we have
\begin{align*}
\forall\,k \in \N\colon\quad
\sum_{i=1}^k a_i &\leq \bigg(\frac{1}{(s+1)\wedge 1} + \frac{1}{(1+t)\wedge 1} \bigg) \max\big\{b\, k^{t+1}, c\, k^{s+1}\big\}\\[1ex]
& = \bigg(\frac{1}{(s+1)\wedge 1} + \frac{1}{(1+t)\wedge 1} \bigg) \max\big\{b, \frac{c}{k^{t-s}}\big\} k^{t+1}.
\end{align*}
Case $-1=s<t$. We first note that the function $x \mapsto (1+\log x)/x^{t+1}$
is bounded from above on $\left[1,+\infty\right[$ by $\max\{1, 1/(t+1)\} = ((t+1)\wedge 1)^{-1}$. Therefore we have for $n\geq 2$
\begin{equation*}
c(1+\log (n-1))_+ \leq c(1 + \log M)
\leq \frac{c}{(t+1)\wedge 1} M^{t+1},
\end{equation*}
while the above inequality is clearly true for $n=1$.
Now, noting that $t-s=t+1$ and recalling the definition of $M$, we have
\begin{equation*}
c(1+\log (n-1))_+ \leq 
\frac{c}{(t+1)\wedge 1} M^{t-s}
= \frac{c}{(t+1)\wedge 1} \frac{c}{b}
= \frac{b}{(t+1)\wedge 1} \frac{c^2}{b^2}.
\end{equation*}
and hence, recalling the second of \eqref{eq:20230830a} we have for $k \geq (c/b)^{1/(t-s)}$,
\begin{equation*}
\sum_{i=1}^k a_i \leq \bigg(\frac{(c/b)^2}{k^{t-s}} +1 \bigg) \frac{b k^{t+1}}{(t+1)\wedge 1}.
\end{equation*}
In the end, by \eqref{eq:20230830b},
since for $k\geq1$, $1+\log k \leq ((t+1)\wedge 1)^{-1} k^{t+1}$, we have
\begin{equation*}
\sum_{i=1}^k a_i \leq 
\frac{k^{t+1}}{(t+1)\wedge 1}\cdot
\begin{cases}
c &\text{if } k<(c/b)^{\frac{1}{t-s}}\\[1ex]
\displaystyle
b\bigg(\frac{(c/b)^2}{k^{t-s}} +1 \bigg)
&\text{if } k\geq(c/b)^{\frac{1}{t-s}}.
\end{cases}
\end{equation*}
Case $s<-1<t$. It follows from \eqref{eq:20230830b}
and \eqref{eq:20230830a} that
\begin{align*}
\sum_{i=1}^k a_i &\leq 
\begin{cases}
\displaystyle
\frac{cs}{s+1} &\text{if } k< (c/s)^{\frac{1}{t-s}}\\
\displaystyle
\frac{cs}{s+1} + \frac{b k^{t+1}}{(t+1)\wedge 1} &\text{if } k \geq (c/s)^{\frac{1}{t-s}}
\end{cases}\\[2ex]
&\leq \frac{cs}{s+1} + \max\bigg\{ b, \frac{c}{k^{t-s}} \bigg\} \frac{k^{t+1}}{(t+1)\wedge 1}.
\end{align*}
The other two cases follow directly from  \eqref{eq:20230830b}
and \eqref{eq:20230830a}. The proof is complete.\qed

\vspace{1ex}
\section{Additional convergence rates}
\label{appC}

\vspace{0.5ex}
In this section we provide a full derivation of the bounds stated in \Cref{rmk:in-expctation} and \Cref{rmk:SGnoise}.

As usual we assume that $f\colon H \to \R$ is a convex and $L$-Lipschitz continuous funtion defined on a Hilbert space and $X\subset H$ is a closed convex and bounded set with diameter $D \geq 0$. 

\subsection{A High-probability bound for standard SsGM under sub-Gaussian noise}

Here we assume that the noise is unbiased and sub-Gaussian, i.e.,
\begin{equation*}
\mathbb{E}[\hat{u}(x,\xi)] = u(x) \in \partial f(x) 
\quad \text{and} \quad
\mathbb{E}[\exp{\left(\|\hat{u}(x,\xi)-u(x)\|^2/\eta^2\right)}] \leq \exp{(1)}
\end{equation*}
We assume $\eta$ to be the smallest value satisfying the above sub-Gaussian condition.
\begin{fact}[Lemma~4.1 in \cite{Lan2020}]
\label{lm:SG_concentration}
Let $(\xi_t)_{1\leq i \leq k}$ be a sequence of random variables  adapted to the filtration $(\mathcal{F}_i)_{0\leq i \leq k}$. Suppose that there exists $(\eta_i)_{1\leq i \leq k} \in \R_{++}^k$ such that for every $i=1,\dots,k$ the followings hold
\begin{align}
\EE[\xi_i|\mathcal{F}_{i-1}] &= 0 \\
\EE[\exp{\left(\xi_i^2/\eta_i^2\right)}|\mathcal{F}_{i-1}] &\leq \exp{(1)} \;.
\end{align}
Then, for every $\lambda \geq 0$,
\[
\mathbb{P} \left(\sum_{i=1}^k \xi_i \geq \lambda \sqrt{\sum_{i=1}^k \eta_i^2} \right) \leq \exp{(-\lambda^2/3)} \;.
\]
\end{fact}
\begin{theorem}
Under the above assumptions, let $\delta \in (0,1)$ and $\gamma_i = \gamma/\sqrt{i}$ with 
\begin{equation}
\label{eq:opt_gamma}
\gamma = \sqrt{\frac{D^2}{4(\eta^2(1+\ln(2/\delta))+L^2)}} \;.
\end{equation} 
then with probability at least $1-\delta$
\begin{equation}
\label{eq:ssgm_sg_rate}
f(\bar{x}_k) - f^* \leq \frac{2}{\sqrt{k}} \left(DL + D \eta \sqrt{\ln\left(\frac{2e}{\delta}\right)} \right).
\end{equation}
\end{theorem}
\begin{proof}
Define $\mathcal{F}_i = \sigma(\hat{u}_1, \dots, \hat{u}_i)$.
We start from the following classical inequality
\[
f(x_i) - f^* \leq \frac{1}{2 \gamma_i}\Big(\|x_i - x^* \|^2 - \|x_{i+1} - x^* \|^2\Big) + \langle \hat{u}_i - u_i, x_i - x^* \rangle + \frac{\gamma_i}{2} \|\hat{u}_i\|^2.
\]
Summing all the inequalities from $1$ up to $k$, proceeding as in \eqref{eq:20240414b}, dividing by $k$, and using Jensen's inequality we get
\begin{align}
\label{eq:error_dec}
f(\bar{x}_k) - f^* & \leq \frac{1}{k} \bigg(\frac{D^2}{2 \gamma_k} + \underbrace{\sum_{\iter=1}^{\Iter} \langle \hat{u}_i - u_i, x^* - x_i \rangle}_{\textbf{A}} + \underbrace{\sum_{\iter=1}^{\Iter} \frac{\gamma_i}{2} \|\hat{u}_\iter\|^2}_{\textbf{B}} \bigg) \;.
\end{align}
Let's start from \textbf{A}. Define $Y_i = \langle \hat{u}_i - u_i, x^* - x_i \rangle$ and note that $\mathbb{E}[Y_i|\mathcal{F}_{i-1}]=0$. Furthermore, notice that
\begin{align*}
\mathbb{E}[\exp{(Y_i^2/(D^2 \eta^2))}|\mathcal{F}_{i-1}] &\leq \mathbb{E}[\exp{(D^2 \|\hat{u}_i-u_i\|^2/(D^2 \eta^2))}|\mathcal{F}_{i-1}] = \mathbb{E}[\exp{(\|\hat{u}_i-u_i\|^2/\eta^2)}|\mathcal{F}_{i-1}] \\
&\leq \exp(1) \;.
\end{align*}
Thus $(Y_i)_{1 \leq i \leq k}$ satisfy the assumptions of \Cref{lm:SG_concentration} with $\eta_i = D \eta$, so it holds
\[
\mathbb{P}\left(A \geq D \eta \lambda \sqrt{k} \right) \leq \exp{(-\lambda^2/3)} \;.
\]
Let $\delta \in (0,1)$, equating the RHS of the above to $\delta$, and solving for $\lambda$ one obtains
\begin{equation}
\label{eq:A_bound}    
\mathbb{P}\left(A \geq \sqrt{3} D \eta \sqrt{k \ln(2/\delta)} \right) \leq \delta/2 \;.
\end{equation}
For $\textbf{B}$ we split the term in
\[
\sum_{\iter=1}^{\Iter} \frac{\gamma_i}{2} \|\hat{u}_i\|^2 \leq \sum_{\iter=1}^{\Iter} \gamma_i \|\hat{u}_i - u_i\|^2 + \sum_{\iter=1}^{\Iter} \gamma_i \|u_i\|^2 \;.
\]
We start from the first term. Notice that defining $\theta_i = \gamma_i/(\sum_{i=1}^k \gamma_i)$, by the convexity of the exponential function it holds
\[
\exp{\left(\sum_{i=1}^k \theta_i (\|\hat{u}_i -u_i \|^2/\eta^2) \right)} \leq \sum_{i=1}^k \theta_i \exp{(\|\hat{u}_i -u_i \|^2/\eta^2)} \;.
\]
Taking expectation in both sides, using the sub-Gaussian assumption and noting that we have $\sum_{i=1}^k \theta_i = 1$, it holds
\begin{equation}
\label{eq:exp_convexity}
\mathbb{E} \left[\exp{\left(\sum_{i=1}^k \theta_i \|\hat{u}_i -u_i \|^2/\eta^2 \right)} \right] \leq 1 \;.
\end{equation}
Thus by applying Markov's inequality and equation \eqref{eq:exp_convexity} one obtains
\begin{align*}
\mathbb{P}\left(\sum_{t=1}^T \gamma_i \|\hat{u}_i-u_i\|^2 \leq (1+\ln(2/\delta)) \eta^2 \sum_{t=1}^T \gamma_i \right) &= \mathbb{P}\left(\sum_{t=1}^T \theta_i \|\hat{u}_i-u_i\|^2/\eta^2 \leq (1+\ln(2/\delta)) \right) \\
&\leq \frac{\mathbb{E} \left[\exp{\left(\sum_{t=1}^T \theta_i \|\hat{u}_i -u_i \|^2/\eta^2 \right)} \right]}{\exp(1+\ln(2/\delta))} \\
&\leq \exp(-\ln(2/\delta)) \\
&= \delta/2 \;.
\end{align*}
In the end, we have
\begin{equation}
\label{eq:B1_bound}
\mathbb{P}\left(\sum_{t=1}^T \gamma_i \|\hat{u}_i-u_i\|^2 \leq (1+\ln(2/\delta)) \eta^2 \sum_{t=1}^T \gamma_i \right) \leq \delta/2 \;.
\end{equation}
For the second term, suing the Lipschitz continuity of $f$ one has
\begin{equation}
\label{eq:B2_bound}
\sum_{\iter=1}^{\Iter} \gamma_i \|u_i\|^2 \leq L^2 \sum_{i=1}^k \gamma_i \;.
\end{equation}
Putting equations \eqref{eq:A_bound}, \eqref{eq:B1_bound}, and  \eqref{eq:B2_bound} together into \eqref{eq:error_dec}, with probability at least $1-\delta$ it holds
\begin{align*}
f(\bar{x}_k) - f^* & \leq \frac{1}{k} \left(\frac{D^2}{2 \gamma_k} + \sqrt{3} D \eta \sqrt{k \ln(2/\delta)} + \left(\eta^2(1+\ln(2/\delta)) + L^2 \right) \sum_{i=1}^k \gamma_i \right).
\end{align*}
If one set $\gamma_i = \gamma/\sqrt{i}$, then the above reduces to
\begin{align*}
f(\bar{x}_k) - f^* & \leq \frac{1}{\sqrt{k}} \left(\frac{D^2}{2} \frac{1}{\gamma} + 2 \left(\eta^2(1+\ln(2/\delta)) + L^2 \right)  \gamma + \sqrt{3} D \eta \sqrt{\ln(2/\delta)} \right).
\end{align*}
Optimizing in $\gamma$, one obtains
\begin{align}
f(\bar{x}_k) - f^* & \leq \frac{1}{\sqrt{k}} \left(2 \sqrt{D^2 \left(\eta^2(1+\ln(2/\delta)) + L^2 \right)} + \sqrt{3} D \eta \sqrt{\ln(2/\delta)} \right),
\end{align}
which is achieved with optimal
\begin{equation}
\gamma=\gamma^* = \sqrt{\frac{D^2}{4(\eta^2(1+\ln(2/\delta))+L^2)}} \;.
\end{equation}
The bound can be simplified at expenses of slightly worse constants as follows
\begin{align}
f(\bar{x}_k) - f^* & \leq \frac{2}{\sqrt{k}} \left(DL + D \eta \sqrt{\ln\left(\frac{2e}{\delta}\right)} \right).
\end{align}
\end{proof}

\subsection{An in-Expectation bound for Clipped-SsGM}

\begin{theorem}
Under the assumptions of \Cref{thm:main} with $p=0$, let $\varepsilon, \beta > 0$ and define $\lambda_i = \max\{\beta \sqrt{i}, (1+\varepsilon)L\}$ and $\gamma_i = \gamma/\sqrt{i}$ with 
\begin{equation}
\gamma = \sqrt{\frac{D^2}{2(\sigma^2/m + L^2)}}.
\end{equation} 
Then, for every $k \in \N$, 
\begin{equation}
\Ex[f(\bar{x}_k) - f^*] \leq 
\sqrt{\frac{D^2(\sigma^2/m + L^2)}{k}} + \frac{\sigma^2}{m} \left(1+\frac{1}{\varepsilon}\right) \sqrt{\frac{D^2}{\beta^2 k}}
\end{equation}
\end{theorem}

\begin{proof}
Starting from \Cref{lemma:ErrorDecomp} with $p=0$, $x=x^*$, and taking expectation in both sides one has
\begin{equation*}
\mathbb{E}[f(\bar{x}_k) - f^*] \leq 
\frac{1}{k}\bigg[
\frac{D^2}{2}
\frac{1}{\gamma_k}+
\underbrace{ \sum_{\iter=1}^{k} \mathbb{E}[\theta_i^b]}_{\textbf{B}} + \frac{1}{2} \underbrace{\sum_{\iter=1}^{k} \mathbb{E}[\nu_i]}_{\textbf{D}}
\bigg]
\end{equation*}
For the term \textbf{B}, by using Cauchy-Schwarz inequality, the tower rule of the expectation, and the bound on the bias of \Cref{rmk:eps}, one gets
\begin{align*}
\textbf{B} &\leq D \frac{\sigma^2}{m} \left(1+ \frac{1}{\varepsilon} \right) \sum_{i=1}^k \frac{1}{\lambda_i} \leq D \frac{\sigma^2}{m \beta} \left(1 + \frac{1}{\varepsilon} \right) \sum_{i=1}^k \frac{1}{\sqrt{i}} \leq  2 D \frac{\sigma^2}{m \beta} \left(1 + \frac{1}{\varepsilon} \right) \sqrt{k}.
\end{align*}
For the term \textbf{D}, one gets
\begin{align*}
\textbf{D} &= \frac{1}{2}\sum_{i=1}^k \gamma_i \mathbb{E}[\mathbb{E}[\|\tilde{u}_i\|^2 | x_1,\dots,x_i]] \leq \frac{1}{2} \left(\frac{\sigma^2}{m} + L^2 \right) \sum_{i=1}^k \gamma_i = \frac{\gamma}{2} \left(\frac{\sigma^2}{m} + L^2 \right) \sum_{i=1}^k \frac{1}{\sqrt{i}} \\
&\leq \gamma \left(\frac{\sigma^2}{m} + L^2 \right) \sqrt{k}
\end{align*}
Putting everything together, one gets
\begin{equation*}
\mathbb{E}[f(\bar{x}_k) - f^*] \leq \frac{1}{k} \bigg[\frac{D^2}{2}
\frac{\sqrt{k}}{\gamma} + 2 D \frac{\sigma^2}{m \beta} \left(1 + \frac{1}{\varepsilon} \right) \sqrt{k} + \gamma \left(\frac{\sigma^2}{m} + L^2 \right) \sqrt{k} \bigg].
\end{equation*}
Optimizing in $\gamma$, one obtains
\begin{equation*}
\Ex[f(\bar{x}_k) - f^*] \leq 
\sqrt{\frac{D^2(\sigma^2/m + L^2)}{k}} + \frac{\sigma^2}{m} \left(1+\frac{1}{\varepsilon}\right) \sqrt{\frac{D^2}{\beta^2 k}},   
\end{equation*}
which is achieved with the optimal
\begin{equation*}
\gamma = \gamma^* = \sqrt{\frac{D^2}{2(\sigma^2/m + L^2)}}.    
\end{equation*}
\end{proof}
\end{document}